\documentclass[11pt]{amsart}

\usepackage{amsmath,amscd,amsfonts,amssymb,overpic}
\usepackage{graphicx}
\usepackage[all]{xy}
\usepackage{color}

\usepackage{times}

\RequirePackage{rotating}                   
    \def\HMt{%
       \setbox0=\hbox{$\widehat{\mathit{HM}}$}
       \setbox1=\hbox{$\mathit{HM}$}
       \dimen0=1.1\ht0
       \advance\dimen0 by 1.17\ht1
       \smash{\mskip2mu\raise\dimen0\rlap{%
          \begin{turn}{180}
              {$\widehat{\phantom{\mathit{HM}}}$}
           \end{turn}} \mskip-2mu    
                \mathit{HM}
   }{\vphantom{\widehat{\mathit{HM}}}}{}}

\newtheorem{thm}[equation]{Theorem}
\newtheorem{defn}[equation]{Definition}
\newtheorem{cor}[equation]{Corollary}
\newtheorem{q}[equation]{Question}

\newtheorem{rmk}[equation]{Remark}
\newtheorem{lemma}[equation]{Lemma}
\newtheorem{claim}[equation]{Claim}

\newtheorem{convention}[equation]{Convention}
\newtheorem{conjecture}[equation]{Conjecture}

\numberwithin{equation}{section}

\newcommand{\R}{\mathbb{R}}
\newcommand{\Z}{\mathbb{Z}}

\newcommand{\F}{\mathbb{F}}

\newcommand{\hh}{\mathfrak{h}}
\renewcommand{\sl}{s \ell}

\newcommand{\op}{\operatorname}
\newcommand{\s}{\vskip.1in}
\newcommand{\n}{\noindent}
\newcommand{\bs}{\boldsymbol}
\newcommand{\bdry}{\partial}

\newcommand{\be}{\begin{enumerate}}
\newcommand{\ee}{\end{enumerate}}

\begin{document}

\title[Sutured HF = sutured ECH]{Sutured Heegaard Floer and embedded contact homologies are isomorphic}

\author{Vincent Colin}
\address{Nantes Universit\'e, CNRS, Laboratoire de Math\'ematiques Jean Leray, F-44000 Nantes, France}
\email{vincent.colin@univ-nantes.fr}
\author{Paolo Ghiggini}
\address{Universit\'e Grenoble Alpes, CNRS, IF, 38000 Grenoble, France}
\email{paolo.ghiggini@univ-grenoble-alpes.fr}
\author{Ko Honda}
\address{University of California, Los Angeles, Los Angeles, CA 90095}
\email{honda@math.ucla.edu} \urladdr{http://www.math.ucla.edu/\char126 honda}

\thanks{VC and PG supported by ANR COSY and the Labex Centre Henri Lebesgue, ANR-11-LABX-0020-01. KH supported by NSF Grants DMS-1549147 and DMS-2003483.}

\date{This version: March 23, 2024.}

\begin{abstract}
We prove the equivalence of the sutured versions of Heegaard Floer homology, monopole Floer homology, and embedded contact homology. As applications we show that the knot versions of Heegaard Floer homology and embedded contact homology are equivalent and that product sutured $3$-manifolds are characterized by the fact that they carry an adapted Reeb vector field without periodic orbits.
\end{abstract}

\maketitle

\section{Introduction}\label{section: introduction}

Heegaard Floer homology, monopole Floer homology and embedded contact homology are three drastically different-looking incarnations of the same closed $3$-manifold invariant. Heegaard Floer homology, introduced by Ozsv\'ath and Szab\'o \cite{OSz1,OSz2}, is easily seen to be topological and admits a combinatorial description via nice  Heegaard diagrams \cite{SW}. Embedded contact homology (ECH), defined by Hutchings \cite{Hu1,Hu2,Hu3} and Hutchings-Taubes \cite{HT1,HT2}, encodes the dynamical properties of an auxiliary Reeb vector field.  Both were defined as symplectic counterparts of monopole Floer homology, defined by Kronheimer and Mrowka \cite{KMbook}. The latter was shown to be isomorphic to ECH by Taubes in \cite{T2,T3,T4,T5,T6} and to Heegaard Floer homology by Kutluhan-Lee-Taubes in \cite{KLT1,KLT2,KLT3,KLT4,KLT5}. Heegaard Floer homology and ECH were independently shown to be isomorphic to each other in \cite{CGH0,CGH1,CGH2,CGH3}.

All three homologies admit natural extensions to compact $3$-manifolds with {\em sutured} boundary \cite{Ju1, CGHH, KM}. Sutured manifolds were introduced by Gabai \cite{Ga} in the context of foliation theory and are now understood to be a bridge between contact geometry and its convex surface theory \cite{Gi} on one hand and geometric decompositions of $3$-manifolds/gauge-theoretic invariants on the other hand. In particular, sutured Heegaard Floer homology, developed under the impulsion of Juh\'asz \cite{Ju}, has striking applications to low-dimensional topology.

Baldwin and Sivek proved in \cite{BS} that the sutured versions of monopole Floer homology and Heegaard Floer homology are isomorphic and that the isomorphism identifies the contact invariants. For what concerns the relation between the sutured versions of Heegaard Floer homology and ECH, we have the following conjecture, which is a slight strengthening of Conjecture 1.5 in \cite{CGHH}.

\begin{conjecture}\label{main conjecture}
  If $(M, \Gamma, \xi)$ is a sutured contact $3$-manifold, then
  $$ECH(M, \Gamma, \xi, A) \simeq SFH(-M, -\Gamma, \mathfrak{s}_\xi + \operatorname{PD}(A))$$
  as relatively graded vector spaces over $\Z/2\Z$, where $A \in H_1(M; \Z)$, $\mathfrak{s}_\xi$ is the canonical Spin$^c$-structure determined by $\xi$, $ECH(M, \Gamma, \xi, A)$ is the sutured ECH of $(M, \Gamma, \xi)$ in the homology class $A$, and $SFH(-M, -\Gamma, \mathfrak{s}_\xi + \operatorname{PD}(A))$ is the sutured Heegaard Floer homology of $(-M,\Gamma)$ in the Spin$^c$-structure $\mathfrak{s}_\xi + \operatorname{PD}(A)$. Moreover the isomorphism identifies the contact invariant of $\xi$ in $ECH(M, \Gamma, \xi, 0)$ to that of $SFH(-M, -\Gamma, \mathfrak{s}_\xi)$.
\end{conjecture}

\begin{rmk} 
ECH admits a unique lifting to the integers defined by a coherent orientation of the moduli spaces defining the boundary map, while Heegaard Floer homology admits different liftings called {\em orientation systems}. In order to state the conjecture over the integers, one would need to identify a canonical orientation system for $SFH$, which has not been done yet. 
\end{rmk}

In this paper we prove part of Conjecture \ref{main conjecture}. In particular, we obtain the following result.

\begin{thm}\label{thm: sutured manifold} 
	Let $(M,\Gamma,\xi)$ be a sutured contact manifold. Then 
\begin{equation} \label{eqn: sutured equivalence}
ECH(M,\Gamma,\xi)\simeq SFH(-M,-\Gamma),
\end{equation}
	where $ECH(M,\Gamma,\xi)$ is the sutured ECH of $(M,\Gamma,\xi)$ summed over all homology classes and $SFH(-M,-\Gamma)$ is the sutured Heegaard Floer homology of $(-M,-\Gamma)$ summed over all relative Spin$^c$-structures.
\end{thm}

For technical reasons we are unable to say anything about the contact invariants of $\xi$ and only prove a partial splitting of Equation~\eqref{eqn: sutured equivalence} into relative Spin$^c$-structures; see Theorem~\ref{thm: sutured manifoldbis} for the precise statement. However, this partial splitting is sufficient to give a complete splitting into relative Spin$^c$-structures in the knot invariant case; see Corollary \ref{cor: knots} and its stronger version Corollary~\ref{cor: isomorphism for knots}.

We give two applications of Theorem \ref{thm: sutured manifold}. The first is the topological invariance of sutured ECH.

\begin{cor} \label{cor: invariance}
  The vector spaces $ECH(M, \Gamma, \xi)$ are topological invariants of $(M, \Gamma)$ (and of the canonical Spin$^c$-structure of $\xi$ if we also take into account the partial decomposition in terms of relative Spin$^c$-structures).
\end{cor}

Previously it was only known that $ECH(M, \Gamma, \xi)$ is an invariant of $(M, \Gamma, \xi)$ by Theorem 10.2.2 in \cite{CGH0} and Theorem 1.2 in \cite{KS}.

As another application of Theorem \ref{thm: sutured manifold}, we characterize product sutured $3$-mani\-folds by the fact that they carry compatible Reeb vector fields without periodic orbits (Theorem~\ref{thm:  characterization}). This extends the proof of the Weinstein conjecture \cite{T1} to contact $3$-manifolds with sutured boundary. We also show that if $(M,\Gamma ,\xi)$ is a taut sutured contact $3$-manifold of depth greater than $2k$ with $H_2 (M)=0$ and if an adapted Reeb vector field $R_\lambda$ is nondegenerate and has no elliptic orbit, then it has at least $k+1$ hyperbolic orbits (Theorem~\ref{thm: depth}).

We also prove an isomorphism of sutured ECH with the sutured version of monopole Floer homology, denoted $SHM$:

\begin{thm}\label{thm: sutured monopole}
  Let $(M, \Gamma, \xi)$ be a sutured contact manifold. Then
  $$ECH(M, \Gamma, \xi) \simeq SHM(-M, - \Gamma).$$
\end{thm}

The proofs of Theorems~\ref{thm: sutured manifold} and~\ref{thm: sutured monopole} go through the construction of a contact embedding of any sutured contact manifold $(M,\Gamma,\xi)$ into a closed contact manifold $(Y,\xi)$, called the {\em contact closure}, on which we control the Reeb dynamics. We abuse notation by using the same name for both contact structures; this is justified by the fact that they agree where they are both defined, i.e., on $M$.  We identify $ECH(M,\Gamma,\xi)$ as a summand in $\widehat{ECH}(Y,\xi)$ and find the analogous identification on the Heegaard Floer side, given by a result of Lekili \cite{Le}.
The isomorphism between the summands then follows from the isomorphism between $\widehat{ECH} (Y,\xi)$ and $\widehat{HF} (-Y)$, proven in the series \cite{CGH0, CGH1, CGH2}.  On the other hand, the closed $3$-manifold $Y$ is the same closure Kronheimer and Mrowka used to define sutured monopole Floer homology, and therefore Theorem \ref{thm: sutured monopole} follows from the computation of $\widehat{ECH}(Y, \xi)$ and Taubes' isomorphism between monopole Floer homology and ECH proven in the series \cite{T2}--\cite{T6}.

Juh\'asz observed that the hat version of knot Floer homology of a knot in a $3$-manifold can be interpreted as the sutured Floer homology of the knot complement with a pair of meridian sutures. Then the isomorphism between the sutured Floer homologies, in its stronger form taking into account the partial splitting according to relative Spin$^c$-structures proved in Theorem \ref{thm: sutured manifoldbis}, can be translated into an isomorphism between knot Floer homology and ECH of a sutured manifold associated to the knot:

\begin{cor}\label{cor: knots}
Let $K$ be a null-homologous knot in a closed manifold $M$ and $S$ a Seifert surface of $K$. If $M(K)$ is the complement of a tubular neighborhood of $K$, $\Gamma_K$ a pair of oppositely oriented disjoint meridians in $\partial M(K)$, and $(M(K), \Gamma_K, \xi)$ a sutured contact manifold, then, for every $d \in \Z$,
\begin{equation} \label{eqn: knot equivalence}
\widehat{HFK}(-M, -K, [S], d) \simeq \bigoplus_{\langle c_1(\mathfrak{s}_{\xi})+ 2 \op{PD}(A), [S] \rangle =2d} ECH(M(K), \Gamma_K, \xi, A).
\end{equation}
\end{cor}

Here $c_1(\mathfrak{s}_\xi) \in H_2(M(K), \partial M(K))$ is the relative Chern class of the canonical Spin$^c$-structure $\mathfrak{s}_{\xi}$. Spano in his thesis \cite{Sp} gave evidence for this isomorphism by showing that the graded Euler characteristic of $SFH(M(K), \Gamma_K, \xi)$ is the Alexander polynomial.  

When $K \subset M$ is a fibered knot and $\xi$ is the Thurston-Winkelnkemper contact structure on $M(K)$, then $ECH(M(K), \Gamma_K, \xi)$ is isomorphic to a version of the periodic Floer homology of the monodromy which will be defined in Section \ref{ssec: PFH}. Thus we have the following corollary of corollary:

\begin{cor}\label{cor: fibred knots}
Let $M$ be a closed manifold and $K$ a fibered knot in $M$ of genus $g$ and fiber $S$. If $\hh$ is an area-preserving representative of the monodromy with zero flux, then
\begin{equation} \label{eqn: knot PFH equivalence}
PFH^\sharp(\hh, d) \simeq \widehat{HFK}(-M, -K, d-g).
\end{equation}
\end{cor}

\begin{rmk} 
It is possible to refine Equations \eqref{eqn: knot equivalence} and \eqref{eqn: knot PFH equivalence} by taking into account the splitting according to relative Spin$^c$-structures; the precise statement will be given in Corollary \ref{cor: isomorphism for knots}.
\end{rmk}

When $d=1$, periodic Floer homology reduces to the usual symplectic Floer homology of a surface automorphism, and therefore Corollary \ref{cor: fibred knots} generalizes previous results of Ni \cite{Ni} and Ghiggini--Spano \cite{GS}. \color{black} The proof here is similar in spirit to that of \cite{Ni}, which goes from the knot to a (different) closed manifold and uses the isomorphism between monopole Floer homology and periodic Floer homology due to Lee-Taubes \cite{LT}, followed by the isomorphism of  \cite{KLT1}--\cite{KLT5}. On the other hand, the proof in \cite{GS} is almost completely independent of the isomorphisms as it uses only the (simpler) open-closed map of \cite{CGH1} and ``standard'' symplectic geometry.

In \cite{KM} Kronheimer and Mrowka defined knot monopole Floer homology groups $HKM(M, K, [S], d)$, where $M$ is a closed manifold, $K \subset M$ a null-homologous knot, $S$ a Seifert surface for $K$, and $d \in \Z$, as the monopole Floer homology of the sutured manifold $(M(K), \Gamma_K)$.
The same argument proving Corollary \ref{cor: knots} also proves the following corollary:

\begin{cor} \label{cor: knot monopole}
Let $K$ be a null-homologous knot in a closed manifold $M$ and $S$ a Seifert surface of $K$. Then, for every $d \in \Z$,
\begin{equation} 
HKM(-M, -K, [S], d) \simeq \bigoplus_{\langle c_1(\mathfrak{s}_{\xi})+ 2 \op{PD}(A), [S] \rangle =2d} ECH(M(K), \Gamma_K, \xi, A).
\end{equation}
\end{cor}

\begin{rmk}
  The reason why sutured monopole Floer homology does not have a decomposition into relative Spin$^c$ summands but knot monopole Floer homology does have an Alexander grading is the same reason why we could not get a full Spin$^c$-decomposition in Theorem \ref{thm: sutured manifold} but we could prove that the isomorphism in Corollary \ref{cor: knots} preserves the Alexander grading.
\end{rmk}

\s\n
{\em Acknowledgements.}  A significant advance in this project was made when the first two authors met at the ``Singular Workshop'', held at the Renyi Institute as part of the Erd\H{o}s Center's semester on singularities and low dimensional topology, and we are grateful to the organizers for the opportunity. We warmly thank Francesco Lin for suggesting to us the proof of Lemma \ref{vanishing of U in HM} which establishes the vanishing of the $U$-map for monopole Floer homology in the relevant Spin$^c$-structures. We thank John Pardon for his question leading to Theorem \ref{thm: characterization} and which motivated us to investigate a sutured version of our isomorphism.  KH is grateful to Yi Ni and the Caltech Mathematics Department for their hospitality during his sabbatical in 2018.

\section{Sutured manifolds and their Floer homologies}

In this section we review some ingredients from \cite{CGHH}, \cite{CGH0} and \cite{Ju1}.   All the Floer-type homology groups will be defined over the ground field $\F= \Z/ 2 \Z$.

\subsection{Balanced sutured manifolds} 

The various sutured invariants mentioned in the introduction are defined for {\em balanced} sutured manifolds, a restricted class of sutured manifolds introduced by Juh\'asz in \cite{Ju1}. Here we present the definition in a slightly modified form because it is convenient for us to present $M$ as a manifold with corners and include the choice of a tubular neighborhood  of the suture in the definition.

\begin{defn}\label{dfn: balanced sutured manifold}
A {\em  balanced} sutured $3$-manifold is a triple $(M,\Gamma, U(\Gamma))$, where $M$ is a compact $3$-manifold with boundary and corners,  $\Gamma$ is an oriented $1$-manifold in $\partial M$ called the {\em suture}, and $U(\Gamma) \simeq [-1,0] \times \Gamma \times [-1,1]$ is a neighborhood of $\Gamma \simeq \{0 \} \times \Gamma \times \{ 0 \}$ in $M$ with coordinates $(\tau,t)\in [-1,0]\times [-1,1]$, such that the following hold:
\begin{itemize}
\item $M$ has no closed components; 
\item $U(\Gamma)\cap \partial M \simeq (\{0 \} \times \Gamma \times [-1,1]) \cup ([-1,0] \times \Gamma \times \{ -1\}) \cup ([-1,0] \times \Gamma \times \{ 1\})$;
\item $\partial M \setminus (\{ 0\}  \times \Gamma \times (-1,1))$ is the disjoint union of two submanifolds which we call $R_- (\Gamma )$ and $R_+(\Gamma )$, where the orientation of $\partial M$ agrees with that of $R_+ (\Gamma)$ and is opposite that of $R_-
(\Gamma )$, and the orientation of $\Gamma$ agrees with the boundary
orientation of $R_\pm(\Gamma)$;
\item the corners of $M$ are precisely $\{ 0\}  \times \Gamma \times \{ \pm 1\}$;
\item $R_{\pm}(\Gamma)$ have no closed components and $\chi(R_-(\Gamma))=\chi(R_+(\Gamma))$. 
\end{itemize}
\end{defn}

\begin{defn}
If $(M,\Gamma,U(\Gamma))$ is a sutured $3$-manifold, $(M,\Gamma,U(\Gamma),\xi)$ is a {\em sutured contact manifold} if there exists a contact form $\lambda$ for $\xi$ with Reeb vector field $R_\lambda$ such that:
\begin{enumerate}
\item[(C1)] $R_\lambda$ is positively transverse to $R_+ (\Gamma)$ and negatively transverse to $R_-(\gamma)$;
\item[(C2)] $\lambda =Cdt+\beta$ on $U(\Gamma)$ for some constant $C>0$, where $\beta$ is independent of $t$. In particular, $R_\lambda =\frac{1}{C} \partial_t$ on $U(\Gamma)$.
\end{enumerate}
A contact form $\lambda$ satisfying (C1) and (C2), and the contact structure $\xi= \ker \lambda$, are said to be {\em adapted to $(M,\Gamma,U(\Gamma))$.}
\end{defn}

From now on, to simplify notation, we will always omit the neighborhood $U(\Gamma)$ in the data associated to a sutured contact manifold. Sometimes we will even regard $M$ as a manifold with (smooth) boundary and $\Gamma$ as a closed codimension-one submanifold with boundary; in such a case it is understood that we introduce convex corners along the boundary of a neighborhood of $\Gamma$.

\subsection{Sutured Floer homology and knot Floer homology} \label{ssec: knot Floer and sutures}

The sutured Heegaard Floer homology $SFH(M,\Gamma)$ of a balanced sutured $3$-manifold $(M,\Gamma)$ is a topological invariant of $(M,\Gamma)$.  It decomposes according to relative $\mbox{Spin}^c$-structures:
$$SFH(M,\Gamma)= \bigoplus_{\frak{s}\in \operatorname{Spin}^c(M, \Gamma)} SFH(M,\Gamma,{\frak s}).$$
We refer to the original paper \cite{Ju1} for the definition.

If $M$ is a closed manifold and $B \subset M$ is a closed ball, we define the balanced sutured manifold $(M_B, \Gamma_B)$, where $M_B= M \setminus \op{int}(B)$ and $\Gamma_B$ is a connected, embedded closed curve in $\partial M_B \simeq S^2$. (In \cite{Ju1} the sutured manifold $(M_B, \Gamma_B)$ is denoted by $M(1)$.) By \cite{Ju1} there is a tautological isomorphism
\begin{equation}\label{eqn: sutured-hat isomorphism in HF}
  \widehat{HF}(M) \simeq SFH(M_B, \Gamma_B).
\end{equation}

When $K$ is a knot in a $3$-manifold $M$, one can form the sutured manifold
$$(M(K), \Gamma_K) = (M\setminus \op{int}(N(K)), \Gamma_K),$$
where $N(K)$ is a tubular neighborhood of $K$ in $M$ and $\Gamma_K$ consists of two disjoint curves parallel to the meridian of $K$ in $\partial N(K)$.
Let $\widehat{HFK}(M,K)$ be the hat version of knot Floer homology defined in \cite{OSz3}. Then by \cite{Ju1} there is a (tautological) isomorphism
\begin{equation}\label{eqn: sutured-knot isomorphism in HF}
  \widehat{HFK}(M,K)\simeq SFH(M(K), \Gamma_K).
\end{equation}

Assume now that $K$ bounds an oriented embedded surface $\Sigma \subset M$. Let $M_0(K)$ be the $3$-manifold obtained by zero-surgery on $M$ along $K$, where the surgery coefficient is computed with respect to the framing induced by $\Sigma$. Then the knot Floer homology group decomposes according to Spin$^c$-structures on $M_0(K)$:
$$\widehat{HFK}(M,K) = \bigoplus_{\tiny\underline{\mathfrak{s}} \in \operatorname{Spin}^c(M_0(K))} \widehat{HFK}(M, K, \underline{\mathfrak{s}}).$$
Let $\widehat{\Sigma} \subset M_0(K)$ be the closed surface obtained by capping off $\Sigma$. Every relative Spin$^c$-structure $\mathfrak{s} \in \operatorname{Spin}^c(M(K), \Gamma_K)$ extends uniquely to a Spin$^c$-structure $\underline{\mathfrak{s}} \in \operatorname{Spin}^c(M_0(K))$ such that
$$\langle c_1(\mathfrak{s}), [\Sigma] \rangle = \langle c_1(\underline{\mathfrak{s}}), [\widehat{\Sigma}] \rangle$$
and Equation \eqref{eqn: sutured-knot isomorphism in HF} can be refined to
\begin{equation}
\widehat{HFK}(M, K, \underline{\mathfrak{s}})\simeq SFH(M(K), \Gamma_K, \mathfrak{s}).
\end{equation}

Finally we recall that one defines, for $d\in \Z$,
$$\widehat{HFK}(M,K, [\Sigma], d) = \bigoplus_{\tiny\begin{array}{c} \underline{\mathfrak{s}} \in \operatorname{Spin}^c(M_0(K)) \\ \langle c_1(\underline{\mathfrak{s}}), [\widehat{\Sigma}] \rangle = 2d \end{array}} \widehat{HFK}(M, K, \underline{\mathfrak{s}}).$$
The integer $d$ is called the {\em Alexander grading}.



\subsection{Sutured ECH}

Let $\lambda$ be a nondegenerate contact form adapted to $(M,\Gamma)$ and $J$ a {\em tailored} almost complex structure from \cite[Section~3.1]{CGHH}.  Since such a $J$ prevents families of holomorphic curves in the symplectization of $M$ from exiting along its boundary \cite[Proposition~5.20]{CGHH}, Hutchings' definition of ECH extends in a straightforward manner to $(M,\Gamma,\lambda,J)$. Just recall here that the sutured ECH chain complex $ECC(M,\Gamma,\lambda,J)$ is generated over $\F$ by orbit sets $\bs\gamma =\{ (\gamma_i,m_i)~|~ i=1,\dots,k; k\in \Z_{\geq 0}\}$ --- this includes the empty set --- where $\gamma_i$ is a simple orbit of the Reeb vector field $R_\lambda$, $m_i \in \Z_{>0}$, and $m_i =1$ whenever $\gamma_i$ is a hyperbolic orbit. We will sometimes write the orbit set $\bs\gamma$ multiplicatively as $\prod_{i}\gamma_i^{m_i}$.  We call $m_i$ the {\em multiplicity} of $\gamma_i$ in $\bs{\gamma}$.

\begin{convention}
In this paper, when we write ``orbit" we mean ``closed/periodic orbit". 
\end{convention}

The coefficient $\langle \partial \bs\gamma,\bs\gamma' \rangle$ in the differential counts ECH index $I=1$ $J$-holo\-morphic curves in the symplectization of $(M,\lambda)$ that are asymptotic to the orbit sets $\bs\gamma$ at $+\infty$ and $\bs\gamma'$ at $-\infty$; see \cite{Hu1}. 
The ECH index $1$ property implies strong restrictions on the asymptotic behavior of a curve approaching an orbit, called {\em partition conditions}, 
for which we refer to \cite[Definitions 4.13 and 4.14 and Theorem~4.15]{Hu2}.
Relying on the analogous result for closed manifolds, we proved in \cite[Theorem~10.2.2]{CGH0} (see also \cite{KS}) that sutured ECH, denoted by $ECH(M,\Gamma,\xi)$, is an invariant of the sutured contact $3$-manifold $(M,\Gamma,\xi)$.  As in the closed case, there exists a  direct sum decomposition into homology classes $A\in H_1 (M;\Z)$ of orbit sets as follows: 
$$ECH(M,\Gamma,\xi)= \bigoplus_{A\in H_1(M;\Z)}ECH(M,\Gamma,\xi,A).$$
If $M$ is a closed manifold, $B \subset M$ a closed ball, $\xi$ is a contact structure that is adapted to $(M_B, \Gamma_B)$ and $A \in H_1(M; \Z) \simeq H_1(M_B; \Z)$, then we define
$$\widehat{ECH}(M, \xi, A) = ECH(M_B, \Gamma_B, \xi, A).$$
The hat version of ECH was originally defined as the mapping cone of a $U$-map, and its equivalence with a sutured ECH was proved in \cite[Theorem 10.3.1]{CGH0}.

\subsection{Periodic Floer homology and sutured ECH} \label{ssec: PFH}

When $K$ is a fibered knot in $M$, the sutured ECH of $(M(K), \Gamma_K)$ can be interpreted as a version of the periodic Floer homology of a special representative of the monodromy of $K$. 

Let $S$ be a fiber of $K$ and let $(\rho, \theta)$ be coordinates on a collar neighborhood $[-1, 0] \times S^1\subset \partial S$ such that $\partial S= \{0 \} \times S^1$. There exist a $1$-form $\lambda$ and a representative $\hh \colon S \stackrel\sim\to S$ of the monodromy such that:
\begin{itemize}
  \item $d \lambda$ is an area form on $S$ and $\lambda = e^\rho d \theta$ near $\partial S$;
  \item $\hh^*\lambda - \lambda$ is exact;
  \item the periodic points of $\hh$ in $\op{int}(S)$ are nondegenerate; and
  \item $\hh|_{\partial S}=\op{id}_{\bdry S}$ and the linearized first return map at every point of $\partial S$ is of the form $\scriptsize\left (\begin{matrix} 1 & 0 \\ a & 1 \end{matrix} \right)$ with $a<0$ (i.e., $\partial S$ is a negative Morse-Bott circle of fixed points).
\end{itemize}
The existence of $\lambda$ and $\hh$ follows from \cite[Lemma 9.3.2]{CGH0} and a standard genericity argument for nondegenerate periodic points.

The mapping torus $N_{(S, \hh)}$ of $(S,\hh)$ carries a suspension flow which is transverse to the fibers and whose first return map is $\hh$. The boundary of $N_{(S, \hh)}$ admits an $S^1$-family of simple orbits of the suspension flow and we choose one orbit that we call $h$. As in the definition of ECH, the {\em periodic Floer homology} chain complex $PFC^\sharp(\hh)$ is generated, as a vector space over $\F$, by orbit sets $\bs\gamma =\{ (\gamma_i,m_i) ~|~ i=1,\dots,k;  k\in \Z_{\geq 0}\}$ (including the empty set), where $\gamma_i$ is a simple orbit of the suspension flow in $\op{int}(N_{(S, \hh)})$ or the orbit $h$ on the boundary,  $m_i \in \Z_{>0}$, and $m_i =1$ whenever $\gamma_i$ is a hyperbolic orbit or $h$ (i.e., $h$ is treated as a hyperbolic orbit, hence the symbol $h$). The name ``periodic Floer homology'' is due to the fact that closed orbits of the suspension flow are in bijection with orbits of periodic points of $\hh$. 

The manifold $N_{(S, \hh)}$ carries a natural stable Hamiltonian structure $(\alpha_0, \omega)$ induced by $d \lambda$ (see \cite[Section 3.1]{CGH1}). Let $J$ be a generic almost complex structure on $\R \times N_{(S, \hh)}$ which is adapted to $(\alpha_0, \omega)$ in the sense of Definition \cite[Definition 3.2.1]{CGH1}. The analytical foundations of ECH go through for stable Hamiltonian structures on mapping tori (see \cite{Hu1} and \cite{LT}) and therefore we define the boundary operator on $PFH^\sharp(\hh)$ by counting $I=1$ $J$-holomorphic maps in $\R \times N_{(S, \hh)}$ asymptotic to orbit sets at the positive and negative ends. Here the situation is less standard than the one considered in \cite{LT} due to the presence of the orbit $h$ belonging to a Morse-Bott family. This situation was treated in detail in \cite[Section 7]{CGH0}, where a similar chain complex $ECC^\sharp(N_{(S, \hh)}, \alpha)$ is defined for a contact form $\alpha$ on $N_{(S, \hh)}$, and the argument goes through unchanged for periodic Floer homology.

Periodic Floer homology splits as a direct sum over homology classes
$$PFH^\sharp(\hh) = \bigoplus_{A \in H_1(N_{(S, \hh)})} PFH^\sharp(\hh, A),$$
as usual. We also define, for $d \in \Z$,
$$PFH^\sharp(\hh, d)= \bigoplus_{A \cdot [S] = d} PFH^\sharp(\hh, A),$$
where $[S]$ is the class of a fiber and $A \cdot [S]$ is the algebraic intersection number.

Note that $N_{(S, \hh)}\simeq M(K)$. We have the following isomorphism.
  
\begin{lemma}\label{ECH and PFH}
Let $\xi$ be a contact structure on $(M(K), \Gamma_K)$ obtained by a small perturbation of the tangent planes of the fibers. Then, for every $A \in H_1(M(K))$,
$$ECH(M(K), \Gamma_K, \xi, A) \simeq PFH^\sharp(\hh, A).$$
\end{lemma} 
  
\begin{proof}
The lemma follows from \cite[Theorem 10.3.2]{CGH0} and the arguments of \cite[Section 3.6]{CGH1}.
\end{proof}

\section{Proofs of Theorems \ref{thm: sutured manifold} and \ref{thm: sutured monopole}}


\subsection{Reduction to connected sutures}

In this subsection we show that we may assume without loss of generality that the suture of $(M, \Gamma, \xi)$ is connected.

\begin{lemma} \label{lemma: reduction to contact sutures} 
If Theorems \ref{thm: sutured manifold} and \ref{thm: sutured monopole} hold for sutured contact manifolds with connected sutures, then they hold for all sutured contact manifolds.
\end{lemma}

\begin{proof}
Let $(M, \Gamma, \xi)$ be a sutured contact manifold with disconnected suture. We glue sutured contact product $1$-handles $(H\times [-1,1], \ker(dt+ \beta))$ to $(M,\Gamma)$, where $t$ is the coordinate of $[-1,1]$ and $\beta$ is a Liouville form on $H$, i.e., we take an {\em interval-fibered extension}, to obtain a sutured contact manifold $(M',\Gamma',\xi')$ with a connected suture $\Gamma'$. From \cite[Section 9]{CGHH} we obtain the isomorphism
$$ECH(M',\Gamma',\lambda') \simeq ECH(M,\Gamma, \lambda).$$
From \cite[Lemma 9.13]{Ju1} we obtain the isomorphism
$$SFH(-M',-\Gamma') \simeq SFH(-M,-\Gamma),$$
since $(-M,-\Gamma)$ is obtained from $(-M',-\Gamma')$ by a sequence of product disk decompositions along the cocores of $H \times [-1,1]$. 
Finally from \cite[Lemma 4.6, Proposition 6.5 and Proposition 6.7]{KM} we obtain the isomorphism
$$SHM(-M', - \Gamma') \simeq SHM(-M, - \Gamma),$$
since there is a product annulus splitting $(-M', -\Gamma')$ into the disjoint union of $(-M, - \Gamma)$ and a product sutured manifold.  
\end{proof}

\subsection{The contact closure} \label{subsection: construction of $Y$}

Let $(M,\Gamma, \xi)$ be a sutured contact $3$-manifold with connected suture. We pick a compact, oriented surface $S$  of genus $g\geq 3$ with connected boundary, together with a $[-1,1]$-invariant contact structure $\xi$\footnote{Since the contact structures will be glued, the contact structures will all be denoted by $\xi$ in this subsection.} on $S\times [-1,1]_t$
such that: 
\begin{itemize}
\item the dividing set of $S \times \{ \pm 1 \}$ consists of a single circle in  $\operatorname{int}(S) \times \{\pm 1 \}$ bounding a disk $D \times \{ \pm 1 \}$;
\item $D \times \{ + 1 \}$ is the negative region of $S \times \{ +1 \}$ and $D \times \{ - 1 \}$ is the positive region of $S \times \{-1 \}$; and
\item  the characteristic foliation, oriented in the usual way, enters $S$ along $\bdry S$.
\end{itemize}
We then glue the product $(S\times [-1,1],\xi)$ to $(M,\Gamma, \xi)$ along $\partial S\times [-1,1] \simeq \Gamma \times [-1,1]$.
We obtain a contact $3$-manifold $(Y_\Sigma,\xi)$ with boundary components 
\begin{gather*}
\Sigma_+ =R_+ \cup_{\partial R_+ \simeq \partial S \times \{1\}} (S\times \{1\}),\\
\Sigma_- =R_- \cup_{\partial R_- \simeq \partial S \times \{-1\}} (S\times \{ -1 \}).
\end{gather*}
Lastly we consider the closed contact $3$-manifold $(Y,\xi)$ obtained by identifying $\Sigma_+$ and $\Sigma_-$ by a $\xi$-compatible diffeomorphism 
$$\widetilde{\psi} \colon \Sigma_+\to \Sigma_-$$
that is the identity
between $S\times \{1 \}$ and $S \times \{- 1\}$. We denote
$$\psi = \widetilde{\psi}|_{R_+(\Gamma)} \colon R_+(\Gamma) \to R_-(\Gamma)$$
and assume that $\psi$ is the identity between $U(\Gamma) \cap R_+(\Gamma)$ and $U(\Gamma) \cap R_-(\Gamma)$.

We let $\Sigma$ be the glued $\Sigma_+=\Sigma_-$ in $Y$, oriented as $\Sigma_+$. Let $e(\xi)$ be the Euler class of $\xi$.  Then
$$\langle e(\xi), [\Sigma] \rangle = \chi(\Sigma)-2.$$
The topological part of such a construction --- turning a sutured manifold into a closed one --- was first considered by Kronheimer and Mrowka in the context of monopole Floer homology \cite{KM}.

The key technical result of this article is the following isomorphism: 

\begin{thm}\label{thm: ECH}
  Let $(M, \Gamma, \xi)$ be a sutured contact $3$-manifold with connected suture and $(Y, \xi)$ its contact closure. Then 
  \begin{equation} \label{eq: ECH}
    \bigoplus_{A \cdot [\Sigma] =1} \widehat{ECH} (Y, \xi, A)\simeq ECH (M,\Gamma, \xi)\oplus ECH (M,\Gamma,\xi)[1].
  \end{equation}
\end{thm}

The proof of this theorem will occupy Section \ref{sec: proof of the main technical theorem}.

\subsection{Proofs of Theorems \ref{thm: sutured manifold} and \ref{thm: sutured monopole} assuming Theorem \ref{thm: ECH}}

We introduce the notation
\begin{gather*}
  \widehat{HF}(-Y|\Sigma) = \bigoplus_{\langle c_1(\mathfrak{s}), [\Sigma]\rangle = \chi(\Sigma)} \widehat{HF}(-Y, \mathfrak{s}), \\
  \widehat{ECH}(Y, \xi|\Sigma) = \bigoplus_{A \cdot [\Sigma]=1} \widehat{ECH}(Y, \xi,A).
\end{gather*}
Similar notation will be used also for $HF^+$ and monopole Floer homology.

\begin{lemma} \label{lemma: e1}
$\widehat{HF}(-Y | \Sigma) \simeq \widehat{ECH} (Y,\xi | \Sigma).$
\end{lemma}

\begin{proof}
Let $\mathfrak{s}_\xi$ be the canonical Spin$^c$-structure determined by $\xi$. Since $\langle e(\xi),[\Sigma]\rangle = \chi(\Sigma)-2$, the map
$$A \mapsto \mathfrak{s}_\xi+ PD(A)$$
gives a bijection between the homology classes satisfying $A \cdot [\Sigma]=1$ and the Spin$^c$-structures satisfying $\langle c_1(\mathfrak{s}), [\Sigma] \rangle = \chi(\Sigma)$.  Finally, by \cite[Theorem 1.2.1]{CGH1} there is an isomorphism
$$\widehat{ECH}(Y, \xi, A) \simeq \widehat{HF}(-Y, \mathfrak{s}_\xi+PD(A)).$$
\vskip-.23in
\end{proof}

Theorem \ref{thm: ECH} provides a link between $ECH(M, \Gamma, \xi)$ and $\widehat{ECH}(Y, \xi | \Sigma)$. In order to prove Theorem \ref{thm: sutured manifold} assuming Theorem \ref{thm: ECH}, it remains to relate $\widehat{HF}(-Y|\Sigma)$ to $SFH(-M, -\Gamma)$.


\begin{lemma}\label{genus minimizing}
If $R_+(\Gamma)$ and $R_-(\Gamma)$ have minimal genus in their relative homology class in $H_2(M, \Gamma)$, then $\Sigma$ has minimal genus in its homology class.
\end{lemma}

\begin{proof}
In this lemma we use Gabai's original definition of sutured manifolds, i.e., we allow empty boundary and empty sutures. Suppose that $R_{\pm}(\Gamma)$ is genus-minimizing in its class in $H_2(M, \Gamma)$. Writing $M=M' \# N$ where $\partial N = \emptyset$ and $M'$ is irreducible, $R_\pm(\Gamma)$ is still genus-minimizing in its class in $H_2(M', \Gamma)$, and therefore $(M',  \Gamma')$ is a taut sutured manifold; see \cite[Definition 2.10]{Ga}.

We have connected sum decompositions $Y_\Sigma= Y_\Sigma' \# N$ and $Y = Y' \# N$. Since $(M', \Gamma)$ is obtained from $Y_\Sigma'$ by a sequence of product annulus and disk decompositions, $Y'_\Sigma$, seen as a sutured manifold with empty suture, is taut by Lemma \cite[Lemma 3.12]{Ga}. Then there is a taut foliation on $Y'$ with $\Sigma$ as a closed leaf (see \cite[Section 5]{Ga}), and therefore $\Sigma$ minimizes the genus in its class in $H_2(Y')$ by the genus-minimizing property of closed leaves in taut foliations; see Corollary 2 of Section 3 of \cite{Th}.  Finally $\Sigma$ also minimizes the genus in $Y$ because any minimal genus surface $\widetilde{\Sigma}$ in the homology class of $\Sigma$ can be made disjoint from the connected sum sphere by an isotopy because it is incompressible.
\end{proof}

\begin{lemma}
$\widehat{HF}(- Y | \Sigma)) \simeq SFH(-M, -\Gamma) \oplus SFH(-M, -\Gamma)[1]$.
\end{lemma} 

\begin{proof}
If $\Sigma$ is not genus-minimizing, then $\widehat{HF}(- Y | \Sigma) =0$ by the adjunction inequality \cite[Theorem 1.6]{OSz2}, together with \cite[Proposition 2.1]{OSz2} and \cite[Theorem 2.4]{OSz2}. On the other hand, if $\Sigma$ is not genus-minimizing, then $R_\pm(\Gamma)$ are not genus-minimizing either by Lemma \ref{genus minimizing}. Then $SFH(-M, -\Gamma)=0$ by \cite[Proposition 9.18]{Ju1} and \cite[Proposition 9.15]{Ju1}. This proves the lemma in the trivial case when $\Sigma$ is not genus-minimizing.
  
When $\Sigma$ is genus-minimizing, \cite[Theorem 24]{Le} shows that
\begin{equation}\label{e3}
HF^+(-Y | \Sigma) \simeq SFH(-M,-\Gamma), 
\end{equation}
and moreover by \cite[Corollary 20]{Le}
\begin{equation}\label{e2}
\widehat{HF}(-Y | \Sigma)\simeq HF^+(-Y | \Sigma) \oplus HF^+(-Y | \Sigma)[1],
\end{equation} because the $U$-map is zero when restricted  to Spin$^c $-structures $\mathfrak{s}$ such that $\langle c_1(\mathfrak{s}), [\Sigma] \rangle= \chi(\Sigma)$.
\end{proof}

\begin{proof}[Proof of Theorem \ref{thm: sutured manifold}]
Theorem \ref{thm: sutured manifold} follows from Lemma \ref{lemma: e1}, Equations \eqref{e3} and \eqref{e2}, and Theorem \ref{thm: ECH}.
\end{proof}

In order to prove Theorem \ref{thm: sutured monopole}, we prove the analogue of \cite[Corollary 20]{Le}, i.e., the vanishing of the $U$-map in the relevant Spin$^c$-structures, for monopole Floer homology. The proof of the following lemma was suggested to us by Francesco Lin.

\begin{lemma}\label{vanishing of U in HM}
Let $Y$ be a closed, connected and oriented $3$-manifold and $\Sigma \subset Y$ an embedded closed, connected, oriented surface of genus at least $2$. Then for every Spin$^c$-structure $\mathfrak{s}$ such that $\langle c_1(\mathfrak{s}), [\Sigma] \rangle = \chi(\Sigma)$, the map
$$U \colon \HMt_\bullet(Y, \mathfrak{s}) \to \HMt_\bullet(Y, \mathfrak{s})$$
is trivial.
\end{lemma}

\begin{proof}
Since $\mathfrak{s}$ is nontorsion, there is an isomorphism
$$\HMt_\bullet(Y, \mathfrak{s}) \simeq \widehat{\mathit{HM}}_\bullet(Y, \mathfrak{s}),$$
because $\overline{\mathit{HM}}_\bullet(Y, \mathfrak{s})=0$, as its definition only involves reducible solutions. Then it suffices to prove that the map
$$U \colon \widehat{\mathit{HM}}_\bullet(Y, \mathfrak{s}) \to \widehat{\mathit{HM}}_\bullet(Y, \mathfrak{s})$$ 
is trivial.

First we consider the case $Y= S^1\times \Sigma$. By \cite[Lemma 2.2]{KM},
$$\widehat{\mathit{HM}}(S^1\times \Sigma | \Sigma) \simeq \F,$$
and therefore $U$ is trivial on $\widehat{\mathit{HM}}_\bullet(S^1\times \Sigma, \mathfrak{s})$ for every $\mathfrak{s}$ such that $\langle c_1(\mathfrak{s}), [\Sigma] \rangle = \chi(\Sigma)$. We treat the general case by a cobordism argument. Take $W= [-1,1] \times Y$. Then the map
$$\widehat{\mathit{HM}}(W) \colon \widehat{\mathit{HM}}_\bullet(Y, \mathfrak{s}) \to  \widehat{\mathit{HM}}_\bullet(Y, \mathfrak{s})$$
is the identity. Now let $W_0\subset W$ be the union of (i) a closed tubular neighborhood of $\{ 0 \} \times \Sigma$ contained in $[- \frac 13, \frac 13] \times Y$; (ii) $[-1, - \frac 23] \times Y$; and (iii) a tube (i.e., a neighborhood of an arc) connecting them. Then $W_0$ is a cobordism from $Y$ to $Y \# (S^1 \times \Sigma)$ and $W_1:= W \setminus \op{int}(W_0)$ is a cobordism from $Y \# (S^1 \times \Sigma)$ to $Y$, and moreover $\widehat{\mathit{HM}}(W_1) \circ \widehat{\mathit{HM}}(W_0)= \widehat{\mathit{HM}}(W)= \op{Id}$. Since $U$ commutes with the cobordism maps, to prove the lemma it suffices to prove that $U$ vanishes on $\widehat{\mathit{HM}}(Y \#(S^1 \times \Sigma), \mathfrak{s}_{\#})$, where $\mathfrak{s}_{\#}$ is the restriction to $Y \#(S^1 \times \Sigma)$ of the Spin$^c$-structure on $W$ induced by $\mathfrak{s}$. If the connected sum is performed along balls that do not intersect $\Sigma' = \{ \theta \} \times \Sigma \subset S^1 \times \Sigma$, then $\langle c_1(\mathfrak{s}_{\#}), [\Sigma'] \rangle = \chi(\Sigma')$, and the vanishing of $U$ on $\widehat{\mathit{HM}}(Y \#(S^1 \times \Sigma), \mathfrak{s}_{\#})$ follows from Bloom, Mrowka and Ozsv\'ath's connected sum formula (see \cite[Theorem 5]{Lin}) and the vanishing of $U$ on $\widehat{\mathit{HM}}_\bullet(S^1 \times \Sigma, \mathfrak{s}_\#|_{S^1 \times \Sigma})$.
\end{proof}

\begin{proof}[Proof of Theorem \ref{thm: sutured monopole}]
By the definition of sutured monopole Floer homology \cite[Definition 4.3]{KM},
\begin{equation}\label{eqn: definition of SHM}
    SHM(-M, - \Gamma) \simeq \HMt_\bullet(-Y | \Sigma).
\end{equation}
Let $\widetilde{\mathit{HM}}_\bullet(-Y|\Sigma)$ be the cone of $U$-map in $\HMt_\bullet(-Y | \Sigma)$. Then  by Equation \eqref{eqn: definition of SHM} and Lemma \ref{vanishing of U in HM},
$$\widetilde{\mathit{HM}}_\bullet(-Y| \Sigma) \simeq SHM(-M, - \Gamma) \oplus SHM(-M, - \Gamma)[1].$$
By \cite{T2} there is an isomorphism between $\HMt_\bullet(-Y | \Sigma)$ and $ECH(Y, \xi | \Sigma)$ which commutes with the $U$-maps, and therefore induces an isomorphism $\widetilde{\mathit{HM}}_\bullet(-Y | \Sigma) \simeq \widehat{ECH}(Y, \xi | \Sigma)$. The theorem then follows from Theorem \ref{thm: ECH}.
\end{proof}

\section{ECH of the contact closure}\label{sec: proof of the main technical theorem}

\subsection{Construction of a Reeb vector field}\label{ssec: Reeb}

The proof of Theorem \ref{thm: ECH} relies in large part on a careful construction of a Reeb vector field, which is given in this subsection.  We decompose $Y=Y'\cup Y''$, where 
$$Y''=S\times S^1 =S\times ([-1,1]/-1 \sim 1) \quad \mbox{and} \quad Y' = M/(R_+ \stackrel{\psi} \sim R_-).$$ 
The submanifolds $Y'$ and $Y''$ are glued along their torus boundary. The Reeb vector field is constructed in three steps: first we modify the contact form on $Y'$ near the boundary to introduce a ``buffer zone'' which will restrict holomorphic curves from going between $Y'$ and  $Y''$, then we construct a contact form on $Y''$ whose Reeb vector field is Morse-Bott and easy to understand, and finally we perturb the Reeb vector field to make the relevant Reeb orbits nondegenerate.

\subsubsection{The buffer zone}  \label{subsubsection: the buffer zone}

Let $\lambda$ be a contact form on $Y'$, obtained by gluing a contact form adapted to $(M, \Gamma)$. The goal of Section~\ref{subsubsection: the buffer zone} is to make a particular modification to $(Y', \lambda)$ on a collar neighborhood $N\subset Y'$ of $\partial Y'$, which we refer to as ``installing a buffer zone".  

Let $N:=[-1, 1]_s\times T^2_{\phi,t} \subset Y'$ be a collar neighborhood of $\partial Y' = \{ s=1\}$ such that $\phi$ is the coordinate in the  $\Gamma$-direction and $t$ is still the coordinate in the fiber direction. Without loss of generality, we may assume that $\lambda|_{N}=e^s d \phi+ dt$. Choose $\epsilon>0$ small.  On $N$ we consider a contact form $\lambda_0|_N$ of the form 
$$\lambda_0|_N= a(s,t)d\phi +b(s)dt,$$
whose Reeb vector field $R_{\lambda_0|_N}$ is parallel to $-\tfrac{\bdry b}{\bdry s} \partial_\phi+\tfrac{\bdry a}{\bdry s} \partial_t   - \tfrac{\bdry a}{\bdry t} \partial_s$.  Here $a$ and $b$ are chosen such that:
\be
\item[(C1)] The contact condition $b\tfrac{\bdry a}{\bdry s}-a\tfrac{\bdry b}{\bdry s}>0$ holds.  Geometrically this means that along the curve $(a(s,t),b(s))$ for fixed $t$, $((\tfrac{\bdry a}{\bdry s}, \tfrac{\bdry b}{\bdry s}),(a,b))$ is an oriented basis.
\item[(C2)] $a(s,t)=e^s$ for $s$ near $\pm 1$ and $a$ does not depend on $t$ when $s\not\in [-\epsilon, \epsilon]$.
\item[(C3)] $b(s)=1$ for $s$ near $1$ and $b(s)=1+\delta$ for $s$ near $-1$ and $\delta>0$ small.
\item[(C4)] On $s \in [-1, -\epsilon]$ (resp.\ $s \in [\epsilon, 1]$),  as $s$ increases, $(\tfrac{\bdry a}{\bdry s},\tfrac{\bdry b}{\bdry s})$ 
rotates in the clockwise (resp.\ counterclockwise) direction from horizontal to nearly vertical (resp.\ nearly vertical to horizontal). See Figure~\ref{fig: ab-plane}.
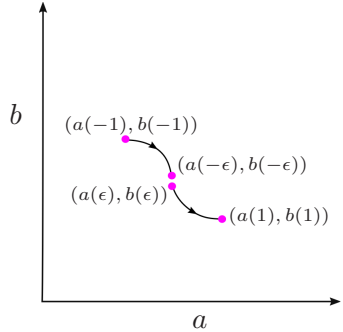
\begin{figure}[ht]
\begin{overpic}[width=4cm]{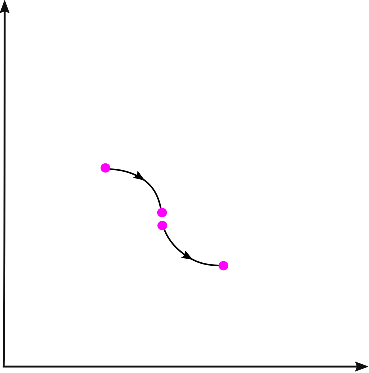}
\put(50,-7){$a$}\put(-10,60){$b$}
\put(8,58){\tiny $(a(-1),b(-1))$} \put(62,28){\tiny $(a(1),b(1))$}
\put(45,45){\tiny $(a(-\epsilon),b(-\epsilon))$} \put(10,35){\tiny $(a(\epsilon),b(\epsilon))$}
\end{overpic}
\caption{The curve $(a(s),b(s))$ on $[-1,-\epsilon]$ and $[\epsilon,1]$.}
\label{fig: ab-plane}
\end{figure}

\item[(C5)] On $[- \epsilon, \epsilon] \times S^1_t$, $a$ is a Morse function $C^1$-close to $1$, with two index one critical points $h_{+,0}$ and $h_{-,0}$, a local maximum $e_{+,0}$, and a local minimum $e_{-,0}$, and whose level sets are drawn in Figure \ref{buffer}, and $b$ satisfies $\tfrac{\bdry b}{\bdry s}<0$. 
\ee
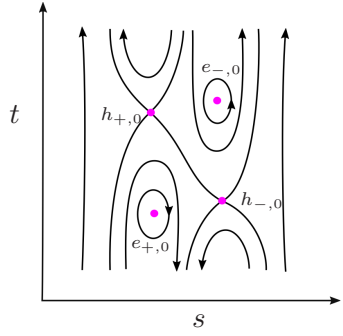
\begin{figure}[ht]
\begin{overpic}[width=4cm]{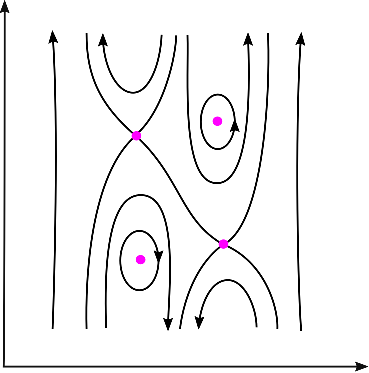}
\put(50,-7){$s$}\put(-10,60){$t$}
\put(20,61){\tiny${h_{+,0}}$} \put(30,19){\tiny${e_{+,0}}$} \put(66,33){\tiny${h_{-,0}}$} \put(53,77){\tiny${e_{-,0}}$}
\end{overpic}
\caption{The buffer zone for $s\in[-\epsilon,\epsilon]$. The level sets of the Morse function $a$ are oriented by the projection of the Reeb vector field to the $(t,s)$-annulus.}
\label{buffer}
\end{figure}

We define the contact form $\lambda_0$ on $Y'$ such that $\lambda_0|_{Y' \setminus N}=(1+\delta)\lambda|_{Y' \setminus N}$ and on $N$ agrees with $\lambda_0|_N$ constructed above. We will refer to $(N,\lambda_0|_N)$ as the {\em buffer zone}.

Next we describe the dynamics of the Reeb vector field $R_{\lambda_0|_N}$ on the buffer zone. Write $Y= \tfrac{\bdry a}{\bdry s} \partial_t   - \tfrac{\bdry a}{\bdry t} \partial_s$ for the $st$-component of $R_{\lambda_0|_N}$. 

\begin{rmk}\label{rmk: tangent to level sets}
Since $da(R_{\lambda_0|_N})=da(Y)=0$, $R_{\lambda_0|_N}$ is tangent to the level sets of $a$.  
\end{rmk}

On $[-\epsilon, \epsilon]\times S^1_t$, the function $a$ has a minimum, a maximum, and two saddle points. By (C5) and Figure~\ref{buffer}, each saddle has a homoclinic connection that gives a closed level line which is positively transverse to $\{t=const\}$ away from the zero of $Y$, and there are two heteroclinic connections between the two saddle points that together form a closed level line which is negatively transverse to $\{t=const\}$ away from the zeros of $Y$. Thus there are $4$ horizontal orbits $e_{-,0}$, $h_{-,0}$, $h_{+,0}$, and $e_{+,0}$ corresponding to the critical points of $a$, where $e_{-,0}$ (resp.\ $e_{+,0}$) corresponds to the minimum (resp.\ maximum) of $a$, and no orbit has negative algebraic intersection with $\{ t=const\}$.

\begin{convention} \label{slope 1}
Given a torus parallel to $T^2_{\phi,t}$ with induced $(\phi,t)$-coor\-dinates, we define the slope of a curve tangent to $q\partial_\phi +p\partial_t$ (or isotopic to such a curve) to be $(q,p)$.
\end{convention}

\begin{lemma}\label{Morse Bott tori in the buffer zone}
There are two families of Morse-Bott tori of slope $(n,1)$ accumulating to the suspension of each homoclinic orbit of $Y$, and no other orbits of the same slope. 
When $s<0$ (resp.\ $s>0$) the Morse-Bott tori are positive (resp.\ negative).
\end{lemma}

\begin{proof}
By Remark~\ref{rmk: tangent to level sets}, $R_{\lambda_0|_N}$ is tangent to the level sets of $a$, viewed as a function on $N$. The closures of the homoclinic trajectories of $h_{+,0}$ and $h_{-,0}$ times $S^1_\phi$ are singular tori $\mathcal{T}_+$ and $\mathcal{T}_-$ that are tangent to $R_{\lambda_0|_N}$.  By (C4), the region between $\{-1\}\times T^2$ and $\mathcal{T}_+$ is foliated by tori, each of which is foliated by $R_{\lambda_0|_N}$ so that the slope rotates clockwise with positive derivative as $s$ increases. Symmetrically, the region between $\mathcal{T}_-$ and $\{1\}\times T^2$ is foliated by tori and the slope induced by $R_{\lambda_0|_N}$ rotates counterclockwise as $s$ increases. Moreover, whenever such a slope is rational, the foliation given by $R_{\lambda_0|_N}$ has an $S^1$-family of (closed) orbits, and the $S^1$-family is Morse-Bott by (C4). The case of slope $(n,1)$ is a special case. 
\end{proof}

The lemma, informally speaking, says that the Reeb vector field in the buffer zone makes (up to a perturbation) a {\em windshield wiper} movement from vertical to horizontal and then to vertical again.

\begin{rmk} \label{rmk: arbitrarily long orbits}
There can be very long orbits that wind around the two horizontal elliptic orbits $e_{-,0}$ and $e_{+,0}$ and have zero intersection number with $\Sigma$. They can be excluded from the ECH chain complex of $Y$ by an easy direct limit argument applied to a sequence of contact forms that do not have horizontal orbits of action $\leq L$ as $L\to \infty$, besides multiples of $e_{-,0},h_{-,0},e_{+,0},h_{+,0}$.
\end{rmk}

\subsubsection{A Morse-Bott contact form on $Y$}

Let $S$ be a compact oriented surface of genus $g\geq 3$ with connected boundary. We pick a closed disk $D_1 \subset S$ and a larger closed disk $D_2$ such that $D_1 \subset \op{int}(D_2)$.  Let $(r,\theta)$ be polar coordinates on $D_2$ so that $D_i=\{r\leq i\}$ for $i=1,2$. We define a contact form
\begin{equation}\label{contact form in the ot zone}
  \lambda_0= g(r)dt+h(r)d \theta
\end{equation}
on $D_2 \times S^1_t$, where $g,h:[0,2]\to \R$ satisfy:
\begin{itemize}
\item $g h'- g' h>0$ (contact condition; the Reeb vector field $R_{\lambda_0}$ is parallel to $h' \partial_t - g' \partial_\theta$);
\item $(g,h)$  makes less than a $\pi$-rotation in a counterclockwise manner from $(g(0),h(0))=(-1,0)$ to $(g(2),h(2))$ with $g(r)=1$ near $r=2$ and $h(2)<0$;
\item $(g'(0),h'(0))=(0,-1)$ and $(g'(2),h'(2))=(0,1)$ ($R_{\lambda_0}$ is vertical at $\{r=0,2\}$); 
\item $(g'(1),h'(1))=(1,0)$ ($R_\lambda$ is horizontal at $\{r=1\}$);
\item $h'g'' - g' h''>0$ (Morse-Bott condition; hence $R_{\lambda_0}$ rotates counterclockwise with nonzero derivative in the basis $(\partial_\theta, \partial_t)$).
\end{itemize}

Then we choose a one-form $\beta$ on $S \setminus \op{int}(D_2)$  and a Morse function $f \colon  S\setminus \operatorname{int}(D_2) \to \R$ close to $1$ such that
\begin{itemize}
\item $d \beta$ is an area form;
\item the Liouville vector field of $\beta$ points into $S$ along $\partial S$ and into $D_2$ along $\partial D_2$;
\item $f$ has a Morse-Bott minimum along $\partial S$;
\item $f$ has $2g$ index one critical points in the interior of $S\setminus \operatorname{int}(D_2)$;
\item $f$ has a Morse-Bott maximum along $\partial D_2$; and 
\item the contact form $f dt + \beta$ agrees with the contact form given by Equation~\eqref{contact form in the ot zone} and the contact form $\lambda_0$ on $Y'$.
\end{itemize}

Then we define $\lambda_0$ on $Y$ by $fdt + \beta$ on $(S \setminus \op{int}(D_2)) \times S^1$, by Equation~\eqref{contact form in the ot zone} on $D_2 \times S^1$ and by $\lambda_0$ on $Y'$.

\begin{convention} \label{slope 2}
Given a torus parallel to $\partial D_2 \times S^1$ with induced $(\theta,t)$-coor\-dinates, we define the slope of a curve tangent to $q\partial_\theta+p\partial_t$ (or isotopic to such a curve) to be $(-q,p)$.
\end{convention}

\subsubsection{Morse-Bott perturbations and excavating the ball}

Fix an unbounded, monotonically increasing sequence of positive real numbers $L_i$, $i \ge 1$, such that $L_i$ is not the period of an orbit of $R_{\lambda_0}$ and fix a sequence of small functions $f_i \colon Y \to \R$ that perturb all the Morse-Bott tori of period less than $L_i$ and slope $(n,1)$ or $(\pm 1, 0)$ (computed with respect to Conventions \ref{slope 1} and \ref{slope 2}) contained in $N\cup (D_2 \setminus \op{int}(D_1)) \times S^1$ into an elliptic-hyperbolic pair of nondegenerate orbits as in \cite[Section 4]{CGH0} (see also \cite{bourgeois:thesis}), and leaving the nondegenerate orbits of period less than $L_i$ unchanged. Then the Reeb vector field $R_{f_i\lambda_0}$ of $f_i\lambda_0$ has two nondegenerate orbits, one elliptic and one hyperbolic, for every Morse-Bott torus of $R_{\lambda_0}$ of period less than $L_i$.

We additionally assume that each $L_i$ is larger than the period of the simple Reeb orbits foliating $\partial D_1 \times S^1$, and that the perturbed orbits $e_0$ and $h_0$ (where $e_0$ is elliptic and $h_0$ is hyperbolic as usual) are supported on $\partial D_1 \times \{ 0 \}$ and $\partial D_1 \times \{ \frac 12 \}$, respectively. The open disk $\op{int}(D_1) \times \{ 0 \}$ is negatively transverse to the flow of $R_{f_i \lambda_0}$ for every $i$, and therefore we can take a sequence\footnote{We cannot choose the ball once and for all $i$ because the support of the perturbations $f_i$ near $\partial D_1 \times S^1$ must shrink as $i$ increases.} of closed balls $B_i$ with concave corners such that $B_{i+1} \subset \op{int}(B_i)$ as follows: we  choose a small solid torus neighborhood $N_i(e_0)$ of $e_0$ whose boundary is tangent to the Reeb flow of $f_i \lambda_0$, and define the ball $B_i$ to be the union of $N_i(e_0)$ together with a very small thickening of the disk $D_1 \times \{0\}$.  
See Figure \ref{figure: ball}.
\begin{figure}[ht]
\begin{overpic}[height=4cm]{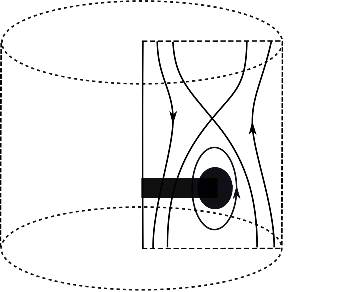}
\put(33.5,29){\tiny $B_i$}
\put(54,50) {\tiny $h_0$}
\end{overpic}
\caption{The concave ball $B_i$ in $D_2\times S^1$, obtained by rotating the shaded region about the vertical central axis.} \label{figure: ball}
\end{figure}

We let $Y_{B_i} := Y\setminus \operatorname{int}(B_i)$ and $\Gamma_{B_i}$ a closed, connected $1$-manifold in $\partial B_i \cap \partial N_i(e_0)$ parallel to $e_0$. Then $(Y_{B_i}, \Gamma_{B_i}, f_i \lambda_0)$ is a sutured contact manifold. We also set $Y''_{B_i}= Y'' \setminus \operatorname{int}(B_i)$.
We make the following observation, which is immediate from the construction:

\begin{claim}\label{claim: no negatively intersecting orbit}
All the Reeb orbits of $f_i\lambda_0$ in $Y_{B_i}$ intersect $\Sigma$ nonnegatively, and the only orbit of $f_i\lambda_0$ in $Y''_{B_i} := Y''\setminus \operatorname{int}(B_i)$ that does not intersect $\Sigma$ is $h_0$.
\end{claim}

Let $ECC^{<L_i}(Y_{B_i}, \Gamma_{B_i}, f_i \lambda_0|\Sigma)$ be the ECH chain complex generated by orbit sets  of total action less that $L_i$ which intersect $\Sigma$ once algebraically, and let  $ECH^{<L_i}(Y_{B_i}, \Gamma_{B_i}, f_i \lambda_0|\Sigma)$ be its homology. By Morse-Bott theory there are canonical inclusions
$$ECC^{<L_i}(Y_{B_i}, \Gamma_{B_i}, f_i \lambda_0 | \Sigma) \hookrightarrow ECC^{<L_j}(Y_{B_j}, \Gamma_{B_i}, f_j \lambda_0 | \Sigma)$$
for $j>i$; see  \cite{CGH0, Yao1, Yao2}. 
We observe that the Morse-Bott correspondence of \cite{Yao1} applies to Morse-Bott cascades of planar holomorphic curves, and this hypothesis is satisfied here by \cite{HS1}.

We define 
$$\mathfrak{C}  = \varinjlim ECC^{<L_i}(Y_{B_i}, \Gamma_{B_i}, f_i \lambda_0 | \Sigma)$$
and denote by $H\mathfrak{C}$ its homology. Since homology commutes with direct limits we have
$$\varinjlim ECH^{< L_i}(Y_{B_i}, \Gamma_{B_i}, f_i \lambda_0| \Sigma) \simeq H\mathfrak{C}.$$
\begin{lemma} \label{lemma: direct limit}
  $H\mathfrak{C}$ is isomorphic to $\widehat{ECH}(Y, \xi | \Sigma)$.
\end{lemma}

\begin{proof}
Fix a reference sutured manifold $(Y_{B_0},\Gamma_{B_0})$, where $B_0$ is a closed ball with concave corners and is a slight enlargement of $B_1$. 
Fix diffeomorphisms $\phi_i \colon (Y_{B_0}, \Gamma_{B_0}) \to (Y_{B_i}, \Gamma_{B_i})$ such that $\phi_i=\op{id}$ outside a fixed small neighborhood of $B_0$ and consider the contact forms $\lambda_i = \phi_i^*(f_i \lambda_0)$. Then
$$ECH^{< L_i}(Y_{B_i}, \Gamma_{B_i}, f_i \lambda_0 | \Sigma) \simeq ECH^{< L_i}(Y_{B_0}, \Gamma_{B_0}, \lambda_i | \Sigma)$$ 
tautologically, and by Lemma 10.2.1 and Corollary 3.2.3 of \cite{CGH0}, we have
$$ECH(Y_{B_0}, \Gamma_{B_0}, \xi | \Sigma) \simeq \varinjlim ECH^{< L_i}(Y_{B_0}, \Gamma_{B_0}, \lambda_i | \Sigma).$$
Hence $ECH(Y_{B_0}, \Gamma_{B_0}, \xi | \Sigma) \simeq H\mathfrak{C}$. Finally 
$$ECH(Y_{B_0}, \Gamma_{B_0}, \xi | \Sigma ) \simeq \widehat{ECH}(Y, \xi | \Sigma)$$ 
by \cite[Theorem 10.3.1]{CGH0}.
\end{proof}

\subsubsection{List of orbits contributing to $\mathfrak{C}$}

We describe a {\em partially defined} trivialization $\tau$ of $\xi$ with respect to which we compute the Conley-Zehnder indices of the orbits. (Here ``partially defined'' means the trivializations do not extend globally to a trivialization of $\xi$.)
\be
\item On $(S\setminus \op{int}(D_2))\times S^1$, $\tau$ comes from the fibration: More specifically, let $\tau'$ be a trivialization of $T(S\setminus \op{int}(D_2))$. Then let $\tau$ be the pullback of $\tau'$ to $\xi$ on $(S\setminus \op{int}(D_2))\times S^1$.
\item On $(\op{int} (D_2)\setminus \op{int} (D_{1/2}))\times S^1$, $\tau$ has first component $-\bdry_r$.
\item On the buffer region $N=[-1,1]_s\times T^2_{\phi,t} \subset Y'$ of $\partial Y' = \{ s=+1\}$, $\tau$ has first component $\bdry_s$. 
\ee


Now we describe the orbits in $Y''$ and in the buffer zone $N$ which contribute to the generators of the chain complex $\mathfrak{C}$. We can regard them equivalently as nondegenerate orbits of the Reeb flow of the perturbed contact form $f_i \lambda_0$ for $i$ sufficiently large, or as possibly degenerate orbits of $R_{\lambda_0}$ via the Morse-Bott correspondence.  In computing the Conley-Zehnder index the first point of view will be taken, while for every other aspect, we will switch from one to the other without mention.   Whenever we say ``orbit'' without further specification, this convention has always to be understood.
For the next several pages we encourage the reader to refer to Figure~\ref{figure: S} for a more graphical description.
\begin{figure}[ht]
\begin{overpic}[height=5.5cm]{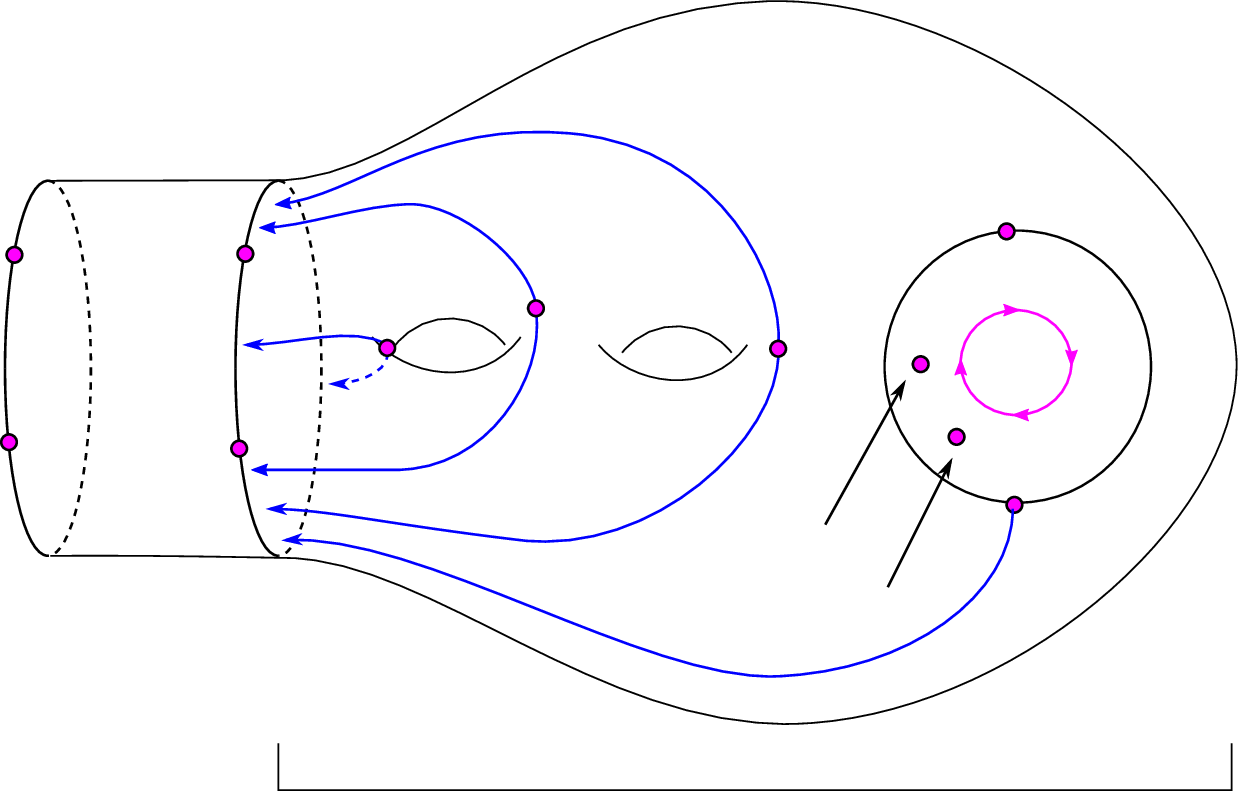}
\put(-5,27.7){\tiny $h_+$}\put(-5,43) {\tiny $e_+$}
\put(14,27.7){\tiny $h_-$}\put(14,43){\tiny $e_-$}
\put(28.4,37.3){\tiny $\delta_1$} \put(44,40){\tiny $\delta_2$}
\put(64,36.7){\tiny $\delta_{2g}$} 
\put(80,47){\tiny $e_2$} \put(82.3,20){\tiny $h_2$}
\put(80,34.2){\tiny $D_1$} \put(90,43){\tiny $\bdry D_2$}
\put(61,20){\tiny $e_{1/n}$} \put(68.5,14){\tiny $h_{1/n}$}
\put(85,38){\tiny $h_0$}
\put(57,-5){\tiny $S\times\{1/2\}$}
\put(50,40){$\cdots$}
\end{overpic}
\vskip.1in
\caption{The orbits that intersect $S\times\{ 0 \}$, given by pink dots.  All the orbits except for $h_0$ intersect $S\times\{0 \}$ once; the orbits besides $h_0,e_{1/n}, h_{1/n}$ are ``vertical'', i.e., parallel to the $S^1$-fibers; the orbit $h_0$ in pink lies on $S\times\{1/2\}$ and bounds $D_1\times\{ 0 \}$. The downward gradient trajectories of $f$  are given in blue.} \label{figure: S}
\end{figure}
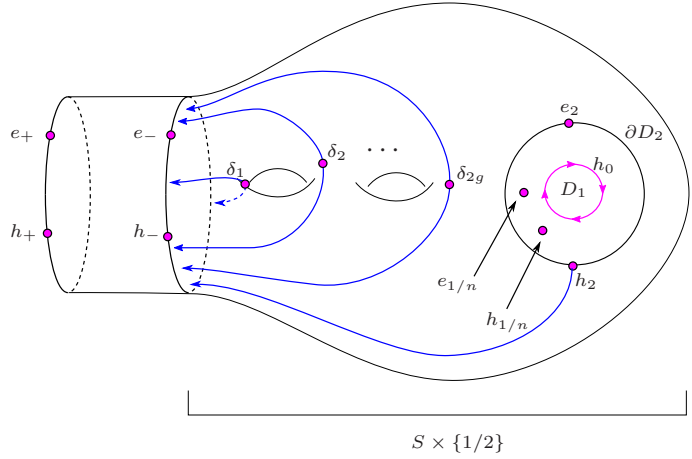
Summarizing the above construction of $\lambda_0$, we have:

\begin{lemma} \label{lemma: list of orbits}
The following is the list of orbits in $Y''$ which can appear in orbit sets that generate $\mathfrak{C}$, where the Conley-Zehnder indices $\mu_{CZ}$ are computed with respect to $\tau$:
\begin{itemize}
\item Over $\operatorname{int}(S\setminus D_2)$, $2g$ vertical hyperbolic orbits $\delta_1,\dots,\delta_{2g}$ of $\mu_{CZ}=0$, where $g$ is the genus of $S$.
\item Over $\partial D_2$, a vertical (i.e., of slope $(0,1)$) $\mu_{CZ}=1$ elliptic orbit $e_2$ and  a $\mu_{CZ}=0$ hyperbolic orbit $h_2$.
\item  For every  $n\in \Z_{>0}$, a $\mu_{CZ}=1$ elliptic orbit $e_{1/n}$  and a $\mu_{CZ}=0$ hyperbolic orbit $h_{1/n}$ in $(\op{int}(D_2) \setminus D_1) \times S^1$, both of slope $(n,1)$  with respect to $(\theta,t)$-coordinates.
\item Over $\bdry D_1$, a $\mu_{CZ}=0$ hyperbolic orbit $h_0$ of slope $(1,0)$.
\end{itemize}
See Figure~\ref{figure: S}.

The following is the list of orbits in the buffer zone $N=[-1,1] \times T^2$ that intersect $\Sigma$ at most once, where the Conley-Zehnder indices $\mu_{CZ}$ are computed with respect to $\tau$:
\begin{itemize}
\item On $\{1 \} \times T^2 = \partial S \times S^1$, a $\mu_{CZ}=-1$ vertical elliptic orbit $e_-$ and a $\mu_{CZ}=0$ vertical hyperbolic orbit $h_-$.
\item On $\{-1\}\times T^2$, a $\mu_{CZ}=1$ vertical elliptic orbit $e_+$ and a $\mu_{CZ}=0$ vertical hyperbolic orbit $h_+$.
\item  For every $n\in \Z_{>0}$, a $\mu_{CZ}=-1$ elliptic orbit $e_{-,1/n}$ and a $\mu_{CZ}=0$ hyperbolic orbit $h_{-,1/n}$  in $(0,1) \times T^2$, both of slope $(n,1)$ with respect to $(\phi,t)$-coordinates.
\item 
 For every $n\in \Z_{>0}$, a $\mu_{CZ}=1$ elliptic orbit $e_{+,1/n}$ and a $\mu_{CZ}=0$ hyperbolic orbit $h_{+,1/n}$ in $(-1,0) \times T^2$, both of slope  $(n,1)$ with respect to $(\phi,t)$-coordinates.
\item Four horizontal (i.e., of slope $(\pm 1, 0)$) orbits $e_{-,0}$, $h_{-,0}$, $h_{+,0}$, and $e_{+,0}$ of $\mu_{CZ}=-1,0,0,1$, respectively; see Figure \ref{buffer}.
\end{itemize}
\end{lemma}

\begin{rmk}
Recall that a Morse-Bott perturbation also creates uncontrollable very long orbits, but they do not contribute to the direct limit, and therefore do not appear in the generators of $\mathfrak{C}$.
\end{rmk}

\subsection{Proof of Theorem \ref{thm: ECH}}

In this section we will regard the holomorphic curves contributing to the differential of $\mathfrak{C}$ either as $J_i$-holomorphic curves in $\R \times Y_{B_i}$ for an almost complex structure $J_i$ which is tailored to $(Y_{B_i}, \Gamma_0, f_i\lambda_0)$ (see \cite[Section~3.1]{CGHH}) for $i$ sufficiently large, or as Morse-Bott cascades consisting of $J$-holomorphic maps  in $\R \times Y$ for an almost complex structure $J$ adapted to $\lambda_0$, augmented by gradient flow trajectories in the Morse-Bott tori. Since the two types of moduli spaces are in canonical bijection by Morse-Bott theory provided that the almost complex structures $J_i$ are chosen to be suitable perturbations of $J$, we will switch from one point of view to the other without explicit mention, very much as we do for Reeb orbits. For this reason the almost complex structure will usually be omitted from the notation.

\subsubsection{The slope}

Topological constraints on holomorphic curves derived from the positivity of intersections will play a central role in the proof of Theorem \ref{thm: ECH}. Here we develop some basic tools. We start by recalling the standard orientation convention for the transverse intersection of two surfaces in a $3$-manifold.

\begin{convention}
If $C$ and $T$ are transversely intersecting oriented surfaces in an oriented $3$-manifold $Z$ (in that order), then at $x \in C \cap T$, let $(n,v)$ be an oriented basis for $T_xC$, where $n$ is an oriented normal to $T$ and $v \in T_x(C\cap T)$. Then $v$ orients $T_x(C\cap T)$.
\end{convention}

\begin{defn}[Slope $\sl(u,T)$]
Let $u:\dot F\to \R\times Z$ be a nontrivial finite energy holomorphic curve in the symplectization of a contact manifold $Z$, $C$ the projection to $Z$ of the image of $u$, and $T\subset Z$ an oriented $2$-torus which:
\be
\item is foliated by (closed) Reeb orbits;
\item does not contain any orbits that the ends of $u$ limit to; and 
\item has a neighborhood which is foliated by tori which in turn are foliated by Reeb orbits. 
\ee
Then the {\em slope $\sl(u,T)$ of $u$ along $T$} is defined as follows: The projection $C$ has a finite number of singularities and, away from those, is an immersion that is transverse to the Reeb vector field. If $T$ does not contain a singular point of the projection, then $C$ is transverse to $T$ and its intersection $C\cap T$ is an immersed oriented $1$-manifold in $T$. 
Its homology class in $H_1(T;\Z)$ is the slope $\sl(u,T)$ of $u$ along $T$.  If $T$ contains a singular point of the projection, then $\sl(u,T)$ is defined as $\sl(u,T')$ for a sufficiently close parallel torus $T'$.
\end{defn}  


\begin{claim}  \label{claim: positively transverse}
The positivity of intersections immediately implies the following:
  \begin{enumerate} 
  \item if $C$ is the projection to $Z$ of the image of $u$, then $C\cap T$ is positively transverse to the Reeb vector field outside of its singular points,  i.e., $(v,R)$ is an oriented basis, where $v$ orients $T(C\cap T)$; and 
  \item if $T$ is foliated by (closed) Reeb orbits in the homology class $\sigma$, then $$\sl(u,T) \cdot \sigma>0.$$
  \end{enumerate}
\end{claim}

See \cite[Section 5.2]{CGH0} for a more detailed account. One should observe that, while the image of a holomorphic curve is canonically oriented, the torus $T$ is not, and the sign of $\sl(u, T)$ depends on the orientation of $T$. However, the sign of an intersection point also depends on the orientation of $T$, and therefore the claim holds for both orientations.

\begin{rmk}
We usually write the slope with respect to a chosen basis of $H_1(T;\Z)$.  We recall that our convention is the following:
  \begin{itemize}
\item if $T$ is parallel to $\bdry Y'$, then we take $(1,0)=-\bdry S$ and $(0,1)$ to be the fiber $[0,1]/\sim$ (Convention \ref{slope 1}); 
\item if $T$ is parallel to $\bdry D_2 \times S^1$ we take $(1,0)= - \bdry D_2$ and $(0,1)$ to be the fiber $S^1= [0,1]/\sim$ (Convention \ref{slope 2}). 
\end{itemize}
\end{rmk}

\subsubsection{Decomposing the differential}

A generator $\bs{\gamma}$ of $\mathfrak{C}$ splits as $\bs{\gamma} = \bs{\gamma}_0 \cup \bs{\gamma}_1$, where $\bs{\gamma}_0$ is an orbit set in $\op{int}(Y')$ and $\bs{\gamma}_1$ is an orbit set in $Y''$. Note that by construction $e_-$ and $h_-$ belong to $Y''$, while every other orbit in the buffer zone $N$ belongs to $\op{int}(Y')$. Since $[\bs{\gamma}] \cdot [\Sigma]=1$ and no orbit intersects $\Sigma$ negatively,\footnote{Those which did have been intercepted by the balls $B_i$.} we are left with two possibilities:
  \begin{itemize}
  \item[(0)]  $[\bs{\gamma}_0] \cdot [\Sigma]=0$ and $[\bs{\gamma}_1] \cdot [\Sigma]=1$, or
  \item[(1)]  $[\bs{\gamma}_0] \cdot [\Sigma]=1$ and $[\bs{\gamma}_1] \cdot [\Sigma]=0$.
  \end{itemize}
  We denote by $\mathfrak{C}_0$ the subspace generated by the orbit sets of type (0) and by $\mathfrak{C}_1$ the subspace  generated by  the orbit sets of type (1). We write the differential $\partial \colon \mathfrak{C} \to \mathfrak{C}$ with respect to the  decomposition $\mathfrak{C}= \mathfrak{C}_0 \oplus \mathfrak{C}_1$ as a matrix
\begin{equation*}
\partial=
\begin{pmatrix}
\partial_{0,0} & \partial_{1,0}\\
\partial_{0,1} & \partial_{1,1}
\end{pmatrix}.
\end{equation*}

For $j=0,1$ we introduce sets ${\mathcal P}_j' \subset H_1(\op{int}(Y'); \Z)$ and ${\mathcal P}_j''  \subset H_1(Y''; \Z) \simeq H_1(Y''_{B_i}; \Z)$ consisting of homology classes $A$ such that $A \cdot [\Sigma]=j$. We denote
$$ECC(- , {\mathcal P}_j^\star) = \bigoplus_{A \in {\mathcal P}_j^\star} ECC(-, A),$$
where ${\mathcal P}_j^\star$ stands for either ${\mathcal P}_j'$ or ${\mathcal P}_j''$.
It is clear that
\begin{gather}
  \nonumber  \mathfrak{C}_0 = \varinjlim \left ( ECC^{<L_i}(\operatorname{int}(Y'), f_i\lambda_0, \mathcal{P}_0')\otimes ECC^{<L_i}(Y''_{B_i}, \Gamma_{B_i}, f_i\lambda_0, \mathcal{P}_1'') \right ), \\
\nonumber \mathfrak{C}_1 =  \varinjlim \left ( ECC^{<L_i} (\operatorname{int}(Y'), f_i \lambda_0, \mathcal{P}_1')\otimes ECC^{<L_i}(Y''_{B_i}, \Gamma_{B_i}, f_i \lambda_0, \mathcal{P}_0'') \right ).
\end{gather}
For the moment the identifications above are only as vector spaces; later we will prove that they are identifications as chain complexes. Inspired by these identifications, we will write $\bs{\gamma}_0 \otimes \bs{\gamma}_1$ for $\bs{\gamma}_0 \cup \bs{\gamma}_1$.

First we prove a preliminary lemma.

\begin{lemma} \label{lemma: outside the tube}
  The only holomorphic curve contributing to the differential of $\mathfrak{C}$ whose projection to $Y$ is not contained in $Y \setminus (\op{int}(D_1) \times S^1)$, after removing all covers of trivial cylinders,  is a holomorphic plane completely contained in $D^1 \times S^1$ and asymptotic to $h_0$.
\end{lemma} 

\begin{proof}
Since $\mathfrak{C}$ has no generators inside $\op{int}(D_1) \times S^1$, the projection of such a curve intersected with $\partial D_1 \times S^1$ must be homologous to a multiple of the meridian. Then by Claim \ref{claim: positively transverse} it must have ends at $h_0$. The only possibility for such a curve to have index $I=1$ is to be a holomorphic plane asymptotic to $h_0$ together, possibly, with covers of trivial cylinders.  Such a holomorphic plane was constructed in \cite[Section 3.1]{We2} in the Morse-Bott setting, and the transition from the Morse-Bott setting to the nondegenerate one is the ``easy case'' of the Morse-Bott correspondence, since there is no need to glue Morse trajectories. 
\end{proof}  

\begin{lemma}\label{lemma: spectral}  
$\partial_{0,1} =0$.
\end{lemma}

\begin{proof} 
We will show that there is no ECH index $1$ holomorphic curve from an orbit set $\bs{\gamma}_0 \otimes \bs{\gamma}_1$ of $\mathfrak{C}_0$ to an orbit set $\bs{\gamma}_0' \otimes \bs{\gamma}_1'$ of $\mathfrak{C}_1$. The element $\bs{\gamma}_1$ is one of the following list:
  \begin{equation} \label{list}
    \begin{matrix}
    e_-, & h_-, & \delta_i, & e_2, & h_2, & e_{1/n}, & h_{1/n}, \\
    e_- h_0, & h_- h_0, & \delta_i h_0, & e_2 h_0, & h_2 h_0, & e_{1/n} h_0, & h_{1/n} h_0,
    \end{matrix}
  \end{equation}
and  the element $\bs{\gamma}_1'$ is either $h_0$ or $\emptyset$. Moreover every holomorphic curve contributing to $\partial_{0,1}$ projects to $Y \setminus (\op{int}(D_1) \times S^1)$ by Lemma \ref{lemma: outside the tube}.
We introduce the notation $\bs{\gamma}_1^\flat$ to denote $\bs{\gamma}_1$ with $e_-$ and $h_-$ removed.

Let $u$ be a $J$-holomorphic curve from $\bs{\gamma}_0\otimes \bs{\gamma}_1$ to $\bs{\gamma}_0'\otimes \bs{\gamma}_1'$. We analyze how the curve $u$ approaches the buffer region in $Y'$ from the $Y''$-side using the following homological argument: Take a torus $T$ parallel to and oriented in the same way as $\partial Y'$ and slightly inside $Y''$;  we may assume that $T$ and its nearby tori are linearly foliated by Reeb orbits by adjusting the construction of the contact form. Let $Z\subset Y''$ be a slight retraction of $Y'' \setminus (\op{int}(D_1) \times S^1)$ obtained by excising the thickened torus between $\partial Y''$ and $T$, and  let $u_Z$ be the projection to $Z$ of the restriction of $u$ to $\R\times Z$. Then $u_Z\cap T$ is homologous to $\bs{\gamma}_1^\flat- \bs{\gamma}_1'$ in $H_1(Z)$ via the surface $u_Z$. Let $b$ be the homology class of $- \partial S$ and $f$ the homology class of the $S^1$-fiber in $H_1(Z) \simeq H_1(S\setminus \op{int}(D_1)) \oplus H_1(S^1)$.  Since $[\bs{\gamma}_1^\flat]= nb+f $, $n\geq 0$, or $b$, or $0$, and $[\gamma_1']= 0$ or $b$, there are five possibilities: 
\begin{enumerate}
\item $\sl(u,T)=(0,1)$, in which case $u$ cannot cross $\partial Y''$ since it is blocked by the vertical flow along $\partial Y''$; see the Blocking Lemma 5.2.3 in \cite{CGH0}. Then $u$ has an end at $e_-$ or $h_-$ and therefore does not contribute to $\partial_{0,1}$. 
\item $\sl(u,T)=(n,1)$, $n\geq 1$, in which case $u$ is either stopped inside the buffer zone by a negative orbit of slope $(n,1)$, or has a negative end at an orbit of slope $(n-k,1)$,  $0< k\leq n$. In the latter case, $\sl(u,\{s_0\}\times T^2)=(k,0)$, where $s_0>0$ is smaller than the $s$-value of the torus foliated by orbits of slope $(n-k,1)$. Then $u$ is blocked in the buffer zone by $k$ orbits of slope $(1,0)$. By this we mean the hyperbolic orbit must have multiplicity at most $1$ but the elliptic orbit can have multiplicity $k$ or $k-1$ and the same number of ends limiting to it.
\item $\sl(u,T)=(1,0)$, in which case $u$ has positive ends at $h_0$ and either at $e_-$ or $h_-$. Then $\sl(u,\{s_1\}\times T^2)=(1,1)$ for $s_1$ slightly smaller than $1$, and $u$ is blocked in the buffer zone by a negative orbit of slope $(1,1)$.
\item $\sl(u,T)=(-1,1)$ or $(-1,0)$, in which case $u$  has a negative end at $h_0$ and no orbits $e_{1/n}, h_{1/n}$ at the positive end. We consider $\sl(u,\{r=2-\epsilon\})$, where $\{r=2-\epsilon\}$ is a torus in $D_2\times S^1$ which is close to the boundary. The slope $\sl(u,\{r=2-\epsilon\})$ must be $(-1,0)$ due to the negative end $h_0$ and the absence of orbits $e_{1/n}, h_{1/n}$ at the positive end, but $(-1,0)$ is not positively transverse to the Reeb vector field, contradicting Claim~\ref{claim: positively transverse}.  
\item $\sl(u, T)=0$, in which case $\bs{\gamma}_1= h_-$, $e_-$, $h_-h_0$, or $e_-h_0$. The trapping lemma \cite[Lemma 5.3.2]{CGH0} implies that either $u$ consists of a holomorphic cylinder from $h_-$ to $e_-$, and therefore does not contribute to $\partial_{0,1}$; or $u_Z$ intersects $T$, and this is incompatible with $\sl(u, T)=0$ by Claim \ref{claim: positively
transverse}.
\end{enumerate}
Hence we are left with Cases (2) and (3).

We explain how to compute $I_{ECH}(u)$ in Case (3). The projection $C$ of the embedded surface $u(\dot F)$ to $Y\setminus (\op{int}(D_1) \times S^1)$ from $\bs{\gamma}^+=h_0 e_-$ or $h_0h_-$ to an orbit $\bs{\gamma}^-=e_{-,1/1}$ or $h_{-,1/1}$ in $N$ of slope $(1,1)$ is constructed by surgering a horizontal section over an enlargement of $S\setminus \operatorname{int} (D_2)$ together with an annulus. (Surgering with an annulus changes $\chi$ by $-1$.)  Since $C$ is embedded and all the orbits involved are simple,
\begin{equation} \label{eqn: index formula}
I_{ECH}(u)=\op{ind}(u) =-\chi (\dot{F})+2\langle c_1(\xi,\tau),C \rangle +\mu_{CZ}(\bs{\gamma}^+)-\mu_{CZ}(\bs{\gamma}^-).
\end{equation}
Here $c_1 (\xi,\tau)$ is the first Chern class of $\xi$ relative to the trivialization $\tau$.  Then $\chi(\dot F)=-2g-1$, $\langle c_1(\xi, \tau),C \rangle=-2g$, $\mu_{CZ}(h_0)=0$, $\mu_{CZ}(h_-)=0$, $\mu_{CZ}(e_-)=-1$, $\mu_{CZ}(e_{-,1/1})=-1$, $\mu_{CZ}(h_{-,1/1})=0$, and 
$$I_{ECH}(u)\leq (2g+1) -4g -1 -0\leq -2g< 0,$$
since $g\geq 3$. This is a contradiction.

Next we consider Case (2). In this case $C$ goes from the orbit set $\bs{\gamma}^+=h_0e_{1/(n-1)}$, $h_0 h_{1/(n-1)}$, $e_{1/n}$, $h_{1/n}$, or a vertical orbit ($\not=e_-$ or $h_-$) times $h_0$ to the orbit set $\bs{\gamma}^-=e_{-,1/n}$, $h_{-,1/n}$, $e_{-,1/(n-k)}e_{-,0}^{k-1} h_{-,0}$, $h_{-,1/(n-k)}e_{-,0}^{k-1} h_{-,0}$, $e_{-,1/(n-k)}e_{-,0}^k$, or $h_{-,1/(n-k)}e_{-,0}^k$, where $e_{-,0}$ and $h_{-,0}$ are the slope $(1,0)$ orbits.  When $C$ is from $\bs\gamma^+=e_{1/n}$ or $h_{1/n}$ to $\bs\gamma^-=e_{-,1/n}$ or $h_{-,1/n}$, $C$ is an $n$-fold cover of an enlargement of $S\setminus \operatorname{int} (D_2)$ and has $\chi=2ng$.  If $\bs\gamma^+$ is changed to $h_0$ times an orbit, then $C$ is modified by surgering with an annulus. (This changes $\chi$ by $-1$ as before.)  If $\bs\gamma^-$ is changed to $e_{-,1/(n-k)}e_{-,0}^{k-1} h_{-,0}$, $h_{-,1/(n-k)}e_{-,0}^{k-1} h_{-,0}$, $e_{-,1/(n-k)}e_{-,0}^k$, or $h_{-,1/(n-k)}e_{-,0}^k$, then $C$ is modified by adding $k-1$ branch points and surgering with an annulus. (This changes $\chi$ by $-k$.) Even though $e_{-,0}$ may have multiplicity $\geq 0$, Formula~\eqref{eqn: index formula} still holds because all ends at $e_{-,0}$ are simple by the partition condition and the multiples of $e_{-,0}$ still have Conley-Zehnder index $-1$.  The number of ends $l$ satisfies $2\leq l\leq 3+k< 3+n$, $\chi(\dot F)= 2-2ng-l$, $\langle c_1(\xi, \tau),C \rangle=-2ng$, $\mu_{CZ}(\bs\gamma^+)\leq 1$, and $\mu_{CZ}(\bs\gamma^-)\geq -k-1$. 
Thus we have 
$$I_{ECH} (u) \leq (2ng+l-2)-4ng +1-(-k-1) \leq -2ng +2n+3<0,$$ 
since $g\geq 3$. This is also a contradiction.
\end{proof}


A consequence of $\partial_{0,1}=0$ is that $\partial_{0,0}^2= \partial_{1,1}^2=0$ and $\partial_{1,0}$ is a chain map.

\subsubsection{Computation of the homologies of $\mathfrak{C}_0$ and $\mathfrak{C}_1$}

\begin{lemma}\label{lemma: zero} 
$H_*(\mathfrak{C}_1,\bdry_{1,1})=0$
\end{lemma}

\begin{proof}
The elements of $\mathfrak{C}_1$ are linear combinations of elements of the form $\bs{\gamma}\otimes h_0$ or $\bs{\gamma} \otimes \emptyset$. By the proof of Case (4) of Lemma \ref{lemma: spectral}, no holomorphic curve in $Y''_{B_i}$ has $h_0$ as a negative end, and by the argument of Case (3) of Lemma \ref{lemma: spectral}, the only holomorphic curve with a positive end at $h_0$ and no other positive end in $Y''$ is the holomorphic plane over $(D_1\times S^1)\setminus B_0$. Thus we can decompose $\partial_{1,1}$ as
  \begin{align*}
    \partial_{1,1}(\bs{\gamma} \otimes \emptyset) & = \partial' \bs{\gamma} \otimes \emptyset, \\
     \partial_{1,1}(\bs{\gamma} \otimes h_0) & =  \partial' \bs{\gamma} \otimes h_0 + \bs{\gamma} \otimes \emptyset,
  \end{align*}
  where $\partial'$ is the differential in $ECC(\op{int}(Y'), \lambda_0)$. The map $K \colon \mathfrak{C}_1\to \mathfrak{C}_1$ defined by
  $$K(\bs{\gamma} \otimes \emptyset) = \bs{\gamma} \otimes h_0, \quad
  K(\bs{\gamma} \otimes h_0)=0$$
  satisfies $\partial_{1,1} \circ K + K \circ \partial_{1,1} = \op{id}$, and therefore $H_*(\mathfrak{C}_1, \partial_{1,1})=0$
\end{proof}

The following lemma enumerates the holomorphic curves that are involved in the calculation of $H_*(\mathfrak{C}_0,\partial_{0,0})$.

\begin{lemma}\label{lemma: list of holomorphic curves}
The list of all connected $I=1$ holomorphic curves in $\R \times Y''$ with ends in $\mathcal{P}_0''\cup \mathcal{P}_1''$ consists of:
\begin{enumerate}
\item[(A)] Two cylinders each from $\delta_i$ to $e_-$, a cylinder from $h_2$ to $e_-$, a cylinder from $e_2$ to $h_-$, and two cylinders each from $e_2$ to $h_2$ and $h_-$ to $e_-$ that correspond to gradient trajectories of a Morse perturbation of $f$ on $S\setminus \operatorname{int}(D_2)$.
\item[(B)] Two cylinders each from $e_{1/n}$ to $h_{1/n}$ and pairs-of-pants in $\R \times (D_2\setminus \op{int}(D_1))\times S^1$ from $e_2 h_0$ to $e_{1/1}$; $h_2h_0$ to $h_{1/1}$; $e_{1/n}h_0$ to $e_{1/(n+1)}$; and $h_{1/n}h_0$ to $h_{1/(n+1)}$.  The pairs-of-pants all belong to moduli spaces of cardinality $1$ mod $2$ (after quotienting by target $\R$-translations).
\item[(C)] A holomorphic plane over $(D_1\times S^1)\setminus B_0$ with a positive end at $h_0$.
\end{enumerate}
\end{lemma}

\begin{proof}
Let $u$ be a connected holomorphic curve in $\R \times Y''$ with positive ends at ${\mathcal P}_0'' \cup {\mathcal P}_1''$. We first note that $u$ either projects to $(S\setminus \operatorname{int}(D_1))\times S^1$ or to $D_1 \times S^1$ by Lemma \ref{lemma: outside  the tube}, and in the latter case it is a holomorphic plane with a positive end at $h_0$. 
Next we show that if $u$ projects to $Y'' \setminus (D_1 \times S^1)$, then it either projects to $Y'' \setminus  (\op{int}(D_2) \times S^1)$ or to $(D_2 \setminus D_1) \times S^1$: Suppose the projection of $u$ intersects both regions. Since in $Y'' \setminus (\op{int}(D_2) \times S^1)$ we consider only orbits in the homology class of the $S^1$-fiber, $\sl(u, \partial D_2 \times S^1)$ can only be one of $(0,-1), (0,0), (0,1)$. However none of the three values is possible by Claim \ref{claim: positively transverse} because the Reeb vector field on $\partial D_2 \times S^1$ has slope $(0,1)$.

(A), (B), and (C) correspond to $u$ with $I(u)=1$ in $Y'' \setminus  (\op{int}(D_2) \times S^1)$, $(D_2 \setminus D_1) \times S^1$, and $D_1 \times S^1$, respectively.

(A) There exists an adapted almost complex structure $J$ on $\R\times(S\setminus \op{int}(D_2))\times S^1$ such that there is a bijection between gradient trajectories $\delta:\R\to S\setminus \op{int}(D_2)$ of $f$ modulo domain $\R$-translation and finite energy $J$-holomorphic cylinders $Z_\delta$ in $\R\times(S\setminus \op{int}(D_2))\times S^1$ that project to $\op{Im}(\delta)$, modulo target $\R$-translation.

(B) A perturbation of the Morse-Bott torus $\{r=r_n\}$ containing $e_{1/n}$ and $h_{1/n}$ gives the two cylinders from $e_{1/n}$ to $h_{1/n}$. The remaining curves follow from adapting Hutchings-Sullivan \cite[Theorem 3.5]{HS1}, but a few remarks are in place:
\begin{enumerate}
\item[(a)] In \cite[Theorem 3.5]{HS1}, there could be curves with more than one negative puncture or more than two positive punctures (the latter are obtained by the ``double rounding'' operations). Those curves are not considered here because the homology class of their ends is not one we are considering here; see the complete list of orbits given by Lemma~\ref{lemma: list of orbits}. 
\item[(b)] The computation in \cite{HS1} is made for perturbations of negative Morse-Bott tori and the actual computation we use is the dual one from \cite{HS2}.
\item[(c)] The work \cite{HS1} is done in the context of periodic Floer homology and requires a ``$d$-regularity'' assumption on the almost complex structure and Reeb vector field. However in \cite{HS2} the argument is extended to ECH where $d$-regularity is not needed.
\end{enumerate} 

(C) is immediate from the first paragraph of the proof.
\end{proof}

We define
$$\mathfrak{C}_0'' = \varinjlim ECC^{<L_i}(Y''_{B_i}, \Gamma_{B_i}, f_i \lambda_0 , {\mathcal P}_1'').$$
As a vector space it is generated by the orbit sets of the list \eqref{list} and its differential $\partial_{0,0}''$, which is determined by Lemma \ref{lemma: list of holomorphic curves}, is:
\begin{enumerate}
\item $\partial_{0,0}''(\gamma h_0) = \gamma$,  where $\gamma = e_-, h_-, \delta_i$,
\item $\partial_{0,0}''(e_2)= h_-$,
\item $\partial_{0,0}''(h_2) = e_-$, 
\item $\partial_{0,0}''(e_2 h_0)=  e_{1/1}+e_2+h_-h_0$, 
\item $\partial_{0,0}''(h_2 h_0)= h_{1/1}+h_2+e_-h_0$,
\item  $\partial_{0,0}''(e_{1/n} h_0) = e_{1/(n+1)}+e_{1/n}$, 
\item $\partial_{0,0}''(h_{1/n} h_0) = h_{1/(n+1)}+h_{1/n}$,
\end{enumerate}
and vanishes on all other generators.

Let $\partial_{0,0}'$ be the differential on $ECC(\op{int}(Y'), \lambda_0, {\mathcal P}_0')$.  We recall that
$$\mathfrak{C}_0 = ECC(\op{int}(Y'), \lambda_0, {\mathcal P}_0') \otimes \mathfrak{C}_0''$$
as a vector space. In the next lemma we prove that the differential splits.

\begin{lemma} \label{lemma: kunneth} 
For every generator $\bs{\gamma}_0 \otimes \bs{\gamma}_1$ of $\mathfrak{C}_0$ we have
$$\partial_{0,0}(\bs{\gamma}_0 \otimes \bs{\gamma}_1) = \partial_{0,0}'(\bs{\gamma}_0) \otimes \bs{\gamma}_1 + \bs{\gamma}_0 \otimes \partial_{0,0}''(\bs{\gamma}_1).$$
\end{lemma}

\begin{proof}
Let $u$ be a connected holomorphic curve from $\bs{\gamma}_0 \otimes \bs{\gamma}_1$ to $\bs{\gamma}'_0 \otimes \bs{\gamma}'_1$ that contributes to $\partial_{0,0}$. 
We show that the projection of $u$ cannot intersect a torus $T' \subset Y'$ that is parallel to $\partial Y'$, foliated by Reeb orbits, and separates $\bs{\gamma}_0, \bs{\gamma}_0'$ from $\bs{\gamma}_1, \bs{\gamma}_1'$. Suppose this is not the case. Since $[\bs{\gamma}_0] \cdot [\Sigma] = [\bs{\gamma}_0'] \cdot [\Sigma]=0$, we have $\sl(u,T')= (k,0)$, where $k>0$ by the positivity of intersections with the vertical Reeb orbits in $\partial Y'$. Since $h_0$ cannot be at a negative end by an argument similar to that of Case (4) of Lemma~\ref{lemma: spectral}, the remaining possibilities for $\bs{\gamma}_1$ and $\bs{\gamma}_1'$ are:
\be 
\item[(a)] $\bs{\gamma}_1$ consists of $e_{1/n}$ or $h_{1/n}$ where $n>0$ and $\bs{\gamma}_1'$ consists of a vertical orbit, $e_{1/n'}$, or $h_{1/n'}$ where $n>n'$; and
\item[(b)] $\bs{\gamma}_1$ consists of $h_0$ and a vertical orbit, $e_{1/n}$, or $h_{1/n}$ where $n>0$ and $\bs{\gamma}_1'$ consists of a vertical orbit,  $e_{1/n'}$, or $h_{1/n'}$ where $n+1>n'$.
\ee
Both (a) and (b) can be ruled out as in Case (2) of Lemma~\ref{lemma: spectral} by considering the ECH index of $u$.
\end{proof}

Finally we compute $H_*(\mathfrak{C}_0, \partial_{0,0})$:

\begin{lemma}\label{lemma: suture} 
$H_*(\mathfrak{C}_0,\partial_{0,0})= {ECH}(\operatorname{int}(Y'), \lambda_0, \mathcal{P}_0')\otimes \langle [e_{1/1}], [h_{1/1}]\rangle.$
\end{lemma}

\begin{proof} 
By Lemma \ref{lemma: kunneth} and the K\"unneth formula we have
$$H_*(\mathfrak{C}_0, \partial_{0,0}) \simeq ECH(\op{int}(Y'), \lambda_0, {\mathcal P}_0') \otimes H_*(\mathfrak{C}_0'', \partial_{0,0}''),$$
and therefore the proof of the lemma reduces to the computation of $H_*(\mathfrak{C}_0'', \partial_{0,0}'')$.

First we observe that the orbit sets $e_{1/n}, h_{1/n}, e_{1/n}h_0, h_{1/n}h_0$, $n \in \Z_{>0}$, form a subcomplex $(\mathfrak{C}_0''', \partial_{0,0}''')$ with homology $H_*(\mathfrak{C}_0''', \partial_{0,0}''')= \langle [e_{1/1}, h_{1/1}] \rangle$. Note that $[e_{1/1}] =[e_{1/n}]$  and $[h_{1/1}] =[h_{1/n}]$ for every $n\in \Z_{>0}$.
  
The quotient complex $\mathfrak{C}_0''/\mathfrak{C}_0'''$ can be identified with the mapping cone of 
$$h_0 C_*(S\setminus \op{int}(D_2), \partial S) \xrightarrow{h_0 \gamma \mapsto \gamma} C_*(S\setminus \op{int}(D_2), \partial S),$$
where $C_*(S\setminus \op{int}(D_2), \partial S)$ is the Morse complex of a Morse perturbation of the Morse-Bott function $f$. This mapping cone is clearly acyclic because the map $h_0 \gamma \mapsto \gamma$ is an isomorphism. The lemma then follows. 
\end {proof}

\subsubsection{Completion of the proof of Theorem~\ref{thm: ECH}}

We now complete the computation of $\widehat{ECH} (Y,\xi | \Sigma)$. By Lemma \ref{lemma: spectral}, the chain complex $\mathfrak{C}$ can be written as the cone of $\mathfrak{C}_1 \xrightarrow{\bdry_{1,0}} \mathfrak{C}_0$.  Using the corresponding exact sequence on homology and Lemmas \ref{lemma: zero}, \ref{lemma: suture} and \ref{lemma: direct limit} we obtain that:
$$\widehat{ECH}(Y,\xi | \Sigma)\simeq  ECH(\operatorname{int}(Y'), \lambda_0, \mathcal{P}_0')\otimes \langle[e_{1/1}],[h_{1/1}]\rangle.$$

Let $\widetilde{Y}'$ be the manifold obtained by excising a thin collar $\mathcal{C}$ of $\partial Y'$ so that $\partial \widetilde{Y}'$ is foliated by orbits of $R_{\lambda_0}$ of irrational slope. We assume that all the orbits of $R_{\lambda_0}$ in $\mathcal{C}$ intersect $\Sigma$ many times, so that
$$ECH(\op{int}(Y'), \lambda_0, {\mathcal P}_0')\simeq ECH(\widetilde{Y}', \lambda_0, {\mathcal P}_0').$$
We also consider a contact form $\widetilde{\lambda}$ on $\widetilde{Y}'$ obtained from $\lambda$ by a modification on a slight enlargement $\mathcal{C}'$ of $\mathcal{C}$ such that:
\begin{itemize}
\item $R_{\widetilde{\lambda}}=R_{\lambda_0}$ near $\partial \widetilde{Y}'$;
\item $\widetilde{\lambda}=\lambda$ on $\op{int}(Y')\setminus \op{int}(\mathcal{C}')$  which contains all the orbit sets in ${\mathcal P}_0'$; and
\item all Reeb orbits in $\op{int}(\mathcal{C}')$ intersect $\Sigma$ many times.
\end{itemize}
Then $ECH(\op{int}(Y'), \lambda, {\mathcal P}_0') \simeq ECH(\widetilde{Y}', \widetilde{\lambda}, {\mathcal P}_0')$. Moreover, 
$$ECH(\widetilde{Y}', \widetilde{\lambda}, {\mathcal P}_0')\simeq ECH(\widetilde{Y}', \lambda_0, {\mathcal P}_0'),$$
by \cite[Proposition 7.2.1]{CGH0} and
$$ECH(\operatorname{int}(Y'), \lambda, \mathcal{P}_0') \simeq ECH (M,\Gamma, \xi)$$ 
by \cite[Theorem 1.9]{CGHH}. Putting the isomorphisms together yields 
$$ECH(\operatorname{int}(Y'), \lambda_0, \mathcal{P}_0') \simeq ECH (M,\Gamma, \xi).$$
In \cite{CGHH}, Theorem 1.9 is proven modulo (i) the invariance of sutured ECH with respect to the contact form and the almost complex structure and (ii) the existence of cobordism maps in sutured ECH that are similar to the ones given by Hutchings-Taubes \cite{HT3} in the closed case. The invariance (i) and the existence of cobordism maps with good properties (see \cite[Section 10.4]{CGHH} for the precise requirements)  (ii) are both given in  \cite[Theorem 10.2.2]{CGH0}. 

Therefore we obtain 
$$\widehat{ECH} (Y,\xi | \Sigma) \simeq ECH (M,\Gamma, \xi)\otimes \langle [e_{1/1}], [h_{1/1}]\rangle,$$
completing the proof of Theorem \ref{thm: ECH}.

\section{Decomposition along $\mbox{Spin}^c$-structures}\label{section: decomposition}

In this section we describe how the isomorphism between sutured Floer homology and sutured ECH behaves with respect to the decomposition along relative $\mbox{Spin}^c$-structures. Let $(M, \Gamma, \xi)$ be a sutured contact manifold  and let $\psi \colon R_+ \to R_-$
be a diffeomorphism which, near the boundary, coincides with the identification induced by the coordinates in the neighborhood $U(\Gamma)$. Let $i_\pm \colon R_\pm \to M$ be the natural inclusions and let $K_\psi \subset H_1(M)$ be given by
$$K_\psi =\operatorname{Im} (i_{-*}\circ \psi_* -i_{+*}).$$
Let $M_\psi : = M/ (x\sim \psi(x))$ be the $3$-manifold with torus boundary obtained by gluing $R_+$ to $R_-$ using $\psi$, and which contains a distinguished surface $R$ corresponding to $R_+$ and $R_-$. Using the Mayer-Vietoris sequence one computes that
$$H_1(M_\psi;\Z) \simeq (H_1(M;\Z)/K_\psi) \oplus \Z,$$
where the $\Z$-factor is generated by a cycle $\gamma$ that intersects $R$ once. 

\begin{thm}\label{thm: sutured manifoldbis}
Let $(M,\Gamma,\xi)$ be a sutured contact $3$-manifold and $\psi : R_+ \to R_-$ a  diffeomorphism as above. Then, for every $A \in H_1(M; \Z)$,
$$\bigoplus_{c \in A+ K_\psi} ECH(M,\Gamma,\xi, c)\simeq \bigoplus_{c \in A+ K_\psi} SFH(-M,-\Gamma, \mathfrak{s}_\xi+PD(c)).$$
\end{thm}

\begin{proof}
First assume that $\Gamma$ is connected.  Let $(Y_\psi, \xi_\psi)$ be the contact closure of $(M, \Gamma, \xi)$ as defined in Section \ref{subsection: construction of $Y$}. Here it is convenient to record the gluing diffeomorphism $\psi$ in the notation and to distinguish $\xi$ from its extension. For every $A \in H_1(M;\Z)$ we denote by $[A]$ its image in $H_1(M; \Z)/K_\psi$ and define $\overline{A}= [A] + \gamma$. Here we identify $H_1(M_\psi; \Z)$ with its image in $H_1(Y_\psi; \Z)$ because the map induced by the inclusion is injective. An inspection of the proof of Theorem \ref{thm: ECH} gives the following refinement of the isomorphism \eqref{eq: ECH}:
$$\bigoplus_{c \in A+K_\psi}(ECH(M, \Gamma, \xi,c) \oplus ECH(M, \Gamma, \xi,c)[1]) \simeq ECH(Y_\psi, \xi_\psi, \overline{A}).$$
Similarly, by Proposition 19, Corollary 20, Theorem 21, and Theorem 24  of Lekili \cite{Le} (note that in \cite{Le} $M$ is denoted $Y$, while $Y_\psi$ is denoted $Y_n$), we have
   \begin{gather*}
     \bigoplus_{c \in A+K_\psi} (SFH(-M, -\Gamma, \mathfrak{s}_{\xi}+PD(c)) \oplus SFH(-M, -\Gamma, \mathfrak{s}_{\xi}+PD(c))[1]) \\  \simeq ~~~ \widehat{HF}(-Y_\psi, \mathfrak{s}_{\xi_\psi}+ PD(\overline{A})).
   \end{gather*}
Although \cite[Theorem 24]{Le} does not explicitly mention the decomposition of $\widehat{HF}(M,\Gamma)$ along relative $\mbox{Spin}^c$-structures, the proof first uses \cite[Theorem 21]{Le} that keeps track of them, followed by the identification of $SFH(M,\Gamma)$ and the ad hoc homology $QFH'(Y_\psi)$ where they are not carefully tracked. The only thing to point out is that this second step is actually done by an isomorphism between chain complexes that automatically respects $\mbox{Spin}^c$-structures.

Now by the isomorphism for closed manifolds \cite{CGH1},
$$\widehat{HF}(-Y_\psi, \mathfrak{s}_{{\xi}_\psi}+PD(\overline{A}) )\simeq \widehat{ECH}(Y_\psi ,\xi_\psi, \overline{A}),$$
and this concludes the proof of the theorem if $\Gamma$ is connected. The general case is obtained by observing that the isomorphisms in the proof of Lemma \ref{lemma: reduction to contact sutures} behave well with respect to homology classes and relative $\mbox{Spin}^c$-structures.
\end{proof}

We now turn our attention to knot Floer homology.

\begin{cor}\label{cor: isomorphism for knots}
Let $K$ be a null-homologous knot in a closed manifold $M$ and $\xi$ a contact structure that is compatible with the sutured manifold $(M(K), \Gamma_K)$. Then
$$ECH(M(K), \Gamma_K, \xi, A) \simeq \widehat{HFK}(-M, -K, \underline{\mathfrak{s}}_\xi + \operatorname{PD}(i_*(A))),$$
where $A \in H_1(M)$,  $\mathfrak{s}_\xi$ is the canonical relative Spin$^c$-structure of $\xi$, $\underline{\mathfrak{s}}_\xi$ is its extension to a Spin$^c$-structure on $M_0(K)$, and $i_* \colon H_1(M(K)) \to H_1(M_0(K))$ is the isomorphism induced by the inclusion $i \colon M(K) \to M_0(K)$.

If $K$ is fibered and $\hh$ is an area-preserving representative of the monodromy with zero flux, then
$$PFH^\sharp(\hh, A) \simeq \widehat{HFK}(-M, -K, \underline{\mathfrak{s}}_\xi + \operatorname{PD}(i_*(A))).$$
\end{cor}

\begin{proof}
Knot Floer homology can be identified with the sutured Heegaard Floer homology of the knot complement with two meridian sutures. Then the corollary follows from Theorem \ref{thm: sutured manifoldbis} by observing that, when $R_+$ and $R_-$ are annuli, $K_\psi = \{ 0 \}$ for every choice of gluing diffeomorphism $\psi$. The statement about periodic Floer homology follows from that of ECH and Lemma \ref{ECH and PFH}.
\end{proof}

\begin{proof}[Proof of Corollaries~\ref{cor: knots} and~\ref{cor: fibred knots}]
Corollaries~\ref{cor: knots} and~\ref{cor: fibred knots} are  just weaker formulations of Corollary \ref{cor: isomorphism for knots}.
\end{proof}

\begin{proof}[Proof of Corollary \ref{cor: knot monopole}]
  The proof is exactly the same as the proof of Corollary \ref{cor: isomorphism for knots}, which is based only on formal properties that holds also for monopole Floer homology; see \cite[Section 5]{KM}.
\end{proof}

\section{A dynamical characterization of product sutured manifolds}\label{section: sutured product}

In this section we prove dynamical results which were announced in \cite{CH}, as corollaries of Theorem \ref{thm: sutured manifold}.
The following answers a question of Pardon.

\begin{thm}\label{thm: characterization}
If $(M,\Gamma, \xi =\ker \lambda )$ is a taut balanced sutured contact manifold whose Reeb vector field $R_\lambda$ has no orbit, then $(M,\xi)$ is a product tight sutured contact manifold $(S\times [0,1],\xi)$, where $\xi$ is $[0,1]$-invariant.  Moreover, if $S$ is planar and $R_\lambda$ has no orbit, then every orbit of $R_\lambda$ flows from $S\times \{0\}$ to $S\times \{1\}$; in particular $R_\lambda$ has no trapped orbits.
\end{thm}

\begin{proof} 
If $R_\lambda$ has no orbit, then $ECH (M,\Gamma, \lambda )\simeq \F \langle[\emptyset]\rangle$ and hence 
$$HF(-M,-\Gamma )\simeq \F.$$ 
Moreover, by Hofer \cite{Hof} (applied without modification to our sutured situation thanks to the control on holomorphic curves given by \cite[Proposition 5.20]{CGHH}), $M$ is irreducible and $\xi$ is tight. By  \cite[Theorem~9.7]{Ju} and the irreducibility of $M$, $(M,\Gamma)$ is a product sutured manifold $(S\times [0,1], \partial S \times [0,1])$. (We remark that it is also possible to prove this result directly using the theory of end-periodic diffeomorphisms of end-periodic surfaces.)  

Next we show that $\xi$ is $[0,1]$-invariant. We decompose $S\times [0,1]$ along a collection of compression disks of the form $a_1 \times [0,1], \dots,a_k\times [0,1]$, where $\{a_1,\dots,a_k\}$ is a basis of arcs for $S$.  Each circle $\partial (a_i \times [0,1])$ intersects the dividing set $\partial S\times \{\tfrac{1}{2}\}$ in exactly two points, i.e., $(S\times[0,1],\partial S\times[0,1])$ is product disk decomposable.  Hence, by the usual convex surface theory, there is a unique tight contact structure on $(S\times[0,1],\partial S\times[0,1])$, and it is $[0,1]$-invariant.

It remains to prove that the Reeb vector field $R_\lambda$ itself flows from $S\times\{0\}$ to $S\times \{1\}$ when $S$ is planar. We use the well-known technique of foliating $\R\times S\times[0,1]$ by holomorphic curves, due to Eliashberg-Hofer \cite{EH} when $S$ is a disk, and to Wendl \cite{We} when $S$ is a more general planar surface. For that, we embed our product as a part of an open book decomposition. We have a page $S_0$ transverse to the Reeb vector field to start the foliation by holomorphic curves asymptotic to the binding and, even if the contact form is not adapted, there is no possibility of breaking since all orbits intersect the pages positively. Since the Reeb flow is transverse to the foliation, there must be a first return map on $S_0$ and the conclusion follows.
\end{proof}

\begin{q} 
	Can one prove that if there is no orbit in $S\times [0,1]$, then there is also no trapped orbit even when $S$ is not planar?
\end{q}

In the higher-dimensional case, such a normalization theorem does not hold, as shown by Geiges, R\"ottgen and Zehmisch in \cite{GRZ} where they exhibit a situation with trapped orbits without periodic ones in a product sutured contact manifold.

Finally, we relate the Reeb dynamics and the {\em depth} of the sutured manifold, i.e., the minimum number of steps in a sutured hierarchy needed to get to a product sutured manifold. This is also the minimal depth of a supported foliation.

\begin{thm}\label{thm: depth}
If $(M,\Gamma ,\xi =\ker \lambda)$ is a taut balanced irreducible sutured contact manifold of depth greater than $2k$ with $H_2 (M)=0$ and if $R_\lambda$ is nondegenerate and has no elliptic orbit, then it has at least $k+1$ hyperbolic orbits.
\end{thm}

\begin{proof}
Under the hypothesis of the theorem, Juh\'asz \cite[Theorem~4]{Ju2} shows that 
$$\op{rk} HF(-M,-\Gamma) \geq 2^{k+1}.$$ 
By our isomorphism, the ECH chain complex must have rank $\geq 2^{k+1}$. When there are no elliptic orbits, this implies the existence of at least $k+1$ hyperbolic orbits for $R_\lambda$.
\end{proof}

Notice that every Reeb vector field can be perturbed to possess only hyperbolic orbits up to a certain action threshold $L$ \cite[Theorem~2.5.2]{CGH1}, typically a number going to infinity with $L$ whenever there is an elliptic orbit to start with.


\begin{thebibliography}{9999}


\bibitem[BS]{BS}
J.\ Baldwin and S.\ Sivek, {\em On the equivalence of contact invariants in sutured Floer homology theories}, Geom.\  Topol.\ {\bf 25} (2021), 1087--1164.

\bibitem[Bou]{bourgeois:thesis}
F.\ Bourgeois, {\em A Morse-Bott approach to contact homology}, Ph.D.\ thesis, 2002.
  
\bibitem[CGH0]{CGH0}
V.\ Colin, P.\ Ghiggini and K.\ Honda, {\em Embedded contact homology and open book decompositions}, Geom.\ Topol., to appear.

\bibitem[CGH1]{CGH1}
V.\ Colin, P.\ Ghiggini and K.\ Honda, {\em The equivalence of Heegaard Floer homology and embedded contact homology via open book decompositions I},  Publ.\ Math.\ Inst.\ Hautes \'Etudes Sci., to appear.

\bibitem[CGH2]{CGH2}
V.\ Colin, P.\ Ghiggini and K.\ Honda, {\em The equivalence of Heegaard Floer homology and embedded contact homology via open book decompositions II},  Publ.\ Math.\ Inst.\ Hautes \'Etudes Sci., to appear.

\bibitem[CGH3]{CGH3}
V.\ Colin, P.\ Ghiggini and K.\ Honda, {\em The equivalence of Heegaard Floer homology and embedded contact homology via open book decompositions III: from hat to plus},  Publ.\ Math.\ Inst.\ Hautes \'Etudes Sci., to appear.

\bibitem[CGHH]{CGHH}
V.\ Colin, P.\ Ghiggini, K.\ Honda and M.\ Hutchings, \textit{Sutures and contact homology I},  Geom.\ Topol.\ {\bf 15} (2011), 1749--1842.

\bibitem[CH]{CH}
V. Colin and K. Honda, {\it Foliations, contact structures and their interactions in dimension three}, Surveys in Differential Geometry, Surveys in 3-Manifolds Topology and Topology, vol. 25 (2020) p. 71-101.

\bibitem[EH]{EH}
Y.\ Eliashberg and H.\ Hofer, {\em A Hamiltonian characterization of the three-ball}, Differential Integral Equations {\bf 7} (1994), 1303--1324.

\bibitem[FM]{FM}
B. Farb, D. Margalit, {\em A primer on mapping class groups},  Princeton (N.J.), Oxford : Princeton University Press, Princeton mathematical series ; 49  (2012).

\bibitem[Ga]{Ga}
D.\ Gabai, \textit{Foliations and the topology of $3$-manifolds}, J.\ Differential Geom.\ {\bf 18} (1983), 445--503.

\bibitem[GRZ]{GRZ} 
H.\ Geiges, N.\ R\"ottgen and K.\ Zehmisch, {\em Trapped Reeb orbits do not imply periodic ones},  Invent.\ Math.\ {\bf 198} (2014), 211--217.

\bibitem[GS]{GS} 
P.\ Ghiggini and G.\ Spano, \textit{Knot Floer homology of fibred knots and Floer homology of surface diffeomorphisms}, preprint 2022. \texttt{ArXiv:2201.12411}.

\bibitem[Gi]{Gi}
E.\ Giroux, \textit{G\'eom\'etrie de contact:\ de la dimension trois
vers les dimensions sup\'erieures}, Proceedings of the International
Congress of Mathematicians, Vol.\ II (Beijing, 2002),  405--414,
Higher Ed.\ Press, Beijing, 2002.


\bibitem[Hof]{Hof}
H.\ Hofer, \textit{Pseudoholomorphic curves in symplectizations with
applications to the Weinstein conjecture in dimension three},
Invent.\ Math.\ {\bf 114} (1993), 515--563.

\bibitem[Hon]{Hon}
K.\ Honda, \textit{On the classification of tight contact structures I},  Geom.\ Topol.\ {\bf 4} (2000), 309--368. 

\bibitem[HKM]{HKM}
K.\ Honda, W.\ Kazez and G.\ Mati\'c, \textit{Right-veering diffeomorphisms of compact surfaces with boundary}, Invent.\ Math.\ {\bf 169}, 427--449.

\bibitem[Hu1]{Hu1}
M.\ Hutchings, \textit{An index inequality for embedded
pseudoholomorphic curves in symplectizations}, J.\ Eur.\ Math.\
Soc.\ (JEMS) {\bf 4}  (2002), 313--361.

\bibitem[Hu2]{Hu2}
M.\ Hutchings, \textit{The embedded contact homology index
revisited}, New perspectives and challenges in symplectic field
theory, 263--297, CRM Proc. Lecture Notes, 49, AMS, 2009.

\bibitem[Hu3]{Hu3}
M.\ Hutchings, \textit{Lecture notes on embedded contact homology},
in {\em Contact and Symplectic Topology}, Bolyai Soc.\ Math.\ Stud.\
Vol.\ 26 (2004)

\bibitem[HS1]{HS1}
M.\ Hutchings and M.\ Sullivan, \textit{The periodic Floer homology
of a Dehn twist}, Algebr.\ Geom.\ Topol.\ {\bf 5} (2005), 301--354.

\bibitem[HS2]{HS2}
M.\ Hutchings and M.\ Sullivan, \textit{Rounding corners of polygons and the embedded contact homology of $T^3$}, Geom.\ Topol.\ {\bf 10} (2006), 169--266.

\bibitem[HT1]{HT1}
M.\ Hutchings and C.\ Taubes, \textit{Gluing pseudoholomorphic curves along branched covered cylinders I},  J.\ Symplectic Geom.\ {\bf 5} (2007),  43--137.

\bibitem[HT2]{HT2}
M.\ Hutchings and C.\ Taubes, \textit{Gluing pseudoholomorphic curves along branched covered cylinders II},  J.\ Symplectic Geom.\ {\bf  7} (2009), 29--133.

\bibitem[HT3]{HT3}
M.\ Hutchings and C.\ Taubes, \textit{Proof of the Arnold chord conjecture in three dimensions 1},
Math. Res. Lett. {\bf 18} (2011), 295--313.

\bibitem[Ju1]{Ju1}
A.\ Juh\'asz, {\em  Holomorphic disks and sutured manifolds}, Algebr.\ Geom.\ Topol.\ {\bf 6} (2006), 1426--1457.

\bibitem[Ju2]{Ju}
A.\ Juh\'asz, {\em Floer homology and surface decompositions}, Geom.\ Topol.\ {\bf 12} (2008), 299--350.

\bibitem[Ju3]{Ju2}
A.\ Juh\'asz, {\em The sutured Floer homology polytope}, Geom.\ Topol.\ {\bf 14} (2010), 1303--1354.

\bibitem[Ko]{Ko}
A.\ Kotelskiy, {\em Comparing homological invariants for mapping classes of surfaces}, Michigan Math.\ J.\ {\bf 70} (3) (2021), 503--560.

\bibitem[KM1]{KMbook}
P.\ Kronheimer and T.\ Mrowka, {\em Monopoles and three-manifolds}, New Mathematical Monographs {\bf 10}.  Cambridge University Press. 796 p. (2007).
  
\bibitem[KM2]{KM}
P.\ Kronheimer, T.\ Mrowka, {\em Knots, sutures and excision}, J.\ Differential Geom.\ {\bf 84} (2010), 301--364.

\bibitem[KLT1]{KLT1}
C.\ Kutluhan, Y.\ Lee and C.\ Taubes, {\em HF=HM I: Heegaard Floer homology and Seiberg-Witten Floer homology}, Geom.\ Topol.\ {\bf 24} (2020), 2829--2854.

\bibitem[KLT2]{KLT2}
C.\ Kutluhan, Y.\ Lee and C.\ Taubes, {\em HF=HM II: Reeb orbits and holomorphic curves for the ech/Heegaard Floer correspondence}, Geom.\ Topol.\ {\bf 24} (2020), 2855--3012.

\bibitem[KLT3]{KLT3}
C.\ Kutluhan, Y.\ Lee and C.\ Taubes, {\em HF=HM III: Holomorphic curves and the differential for the ech/Heegaard Floer correspondence}, Geom.\  Topol.\ {\bf 24} (2020), 3013--3218.

\bibitem[KLT4]{KLT4}
C.\ Kutluhan, Y.\ Lee and C.\ Taubes, {\em HF=HM IV: The Seiberg-Witten Floer homology and ech correspondence}, Geom.\ Topol.\ {\bf 24} (2020), 3219--3469.

\bibitem[KLT5]{KLT5}
C.\ Kutluhan, Y.\ Lee and C.\ Taubes, {\em HF=HM V: Seiberg-Witten-Floer homology and handle addition}, Geom.\ Topol.\ {\bf 24} (2020), 3471--3748.

\bibitem[KS]{KS}
C.\ Kutluhan and S.\ Sivek, {\em Sutured ECH is a natural invariant}, Mem. Amer. Math. Soc. {\bf 275} (2022), no. 1350, iii+136 pp

\bibitem[Le]{Le}
Y.\ Lekili, {\it Heegaard Floer homology of broken fibrations over the circle}, Adv.\ Math.\ {\bf 244} (2013), 268--302. 

\bibitem[Lin]{Lin}
F.\ Lin, \textit{$\op{Pin}(2)$-monopole Floer homology, higher compositions and connected sums}, J.\ Topol.\ {\bf 10} (2017), 921--969. 

\bibitem[Li]{Li}
R.\ Lipshitz, \textit{A cylindrical reformulation of Heegaard Floer homology}, Geom.\ Topol.\ {\bf 10} (2006), 955--1097.

\bibitem[LT]{LT}
Y.-J.\ Lee and C.\ Taubes, {\em Periodic Floer homology and Seiberg-Witten-Floer cohomology}, J.\ Symplectic Geom.\ {\bf 10} (2012), 81--164.
  
\bibitem[Ni]{Ni}
Y.\ Ni, \textit{The next-to-top term in knot Floer homology}, Quantum Topol.\ {\bf 13} (2022), 579--591. 

\bibitem[OSz1]{OSz1}
P.\ Ozsv\'ath and Z.\ Szab\'o, \textit{Holomorphic disks and
topological invariants for closed three-manifolds}, Ann.\ of Math.\
(2) {\bf 159} (2004), 1027--1158.

\bibitem[OSz2]{OSz2}
P.\ Ozsv\'ath and Z.\ Szab\'o, \textit{Holomorphic disks and
three-manifold invariants: properties and applications},  Ann.\ of
Math.\ (2) {\bf 159}  (2004), 1159--1245.

\bibitem[OSz3]{OSz3}
P.\ Ozsv\'ath and Z.\ Szab\'o, \textit{Holomorphic disks and knot invariants}, Adv.\ Math.\ {\bf 186} (2004),  58--116. 
  
\bibitem[OSz4]{OSz4}
P.\ Ozsv\'ath and Z.\ Szab\'o, \textit{Holomorphic disks, link invariants and the multi-variable Alexander polynomial}, Algebr.\ Geom.\ Topol.\ {\bf 8} (2008), 615--692.

\bibitem[SW]{SW}
S.\ Sarkar and J.\ Wang, \textit{An algorithm for computing some Heegaard Floer homologies}, Ann.\ of Math.\ (2) {\bf 171} (2010), 1213--1236.

\bibitem[Sp]{Sp} 
G.\ Spano, {\em Knots invariants in embedded contact homology}, Ph.D.\ thesis, Hal tel-01085021.

\bibitem[T1]{T1}
C.\ Taubes, \textit{The Seiberg-Witten equations and the Weinstein conjecture}, Geom.\ Topol.\ {\bf 11} (2007), 2117--2202.

\bibitem[T2]{T2}
C.\ Taubes, \textit{Embedded contact homology and Seiberg-Witten Floer cohomology I}, Geom.\ Topol.\ {\bf  14} (2010), 2497--2581.

\bibitem[T3]{T3}
C.\ Taubes, \textit{Embedded contact homology and Seiberg-Witten Floer cohomology II}, Geom.\ Topol.\ {\bf  14} (2010), 2583--2720.

\bibitem[T4]{T4}
C.\ Taubes, \textit{Embedded contact homology and Seiberg-Witten Floer cohomology III}, Geom.\ Topol.\ {\bf  14} (2010), 2721--2817.

\bibitem[T5]{T5}
C.\ Taubes, \textit{Embedded contact homology and Seiberg-Witten Floer cohomology IV}, Geom.\ Topol.\ {\bf  14} (2010), 2819--2960.

\bibitem[T6]{T6}
C.\ Taubes, \textit{Embedded contact homology and Seiberg-Witten Floer cohomology III}, Geom.\ Topol.\ {\bf  14} (2010), 2961--3000.

\bibitem[Th]{Th}
W.\ P.\ Thurston, \textit{A norm on the homology of $3$-manifolds}, Mem.\ Am.\ Math.\ Soc.\ {\bf 339}, (1986), 99--130.
  
\bibitem[We]{We}
C.\ Wendl, \textit{Open book decompositions and stable Hamiltonian structures}, Expo.\ Math.\ {\bf 28} (2010), 187--199.

\bibitem[We2]{We2}
C.\ Wendl, \textit{Finite energy foliations and surgery on transverse links}, Ph.D.\ thesis, 2005.

\bibitem[Yao1]{Yao1}
Y.\ Yao, \textit{From cascades to $J$-holomorphic curves and back}, preprint 2022. \texttt{arxiv:2206.04334}.

\bibitem[Yao2]{Yao2}
Y.\ Yao, \textit{Computing embedded contact homology in Morse-Bott settings}, preprint 2022.  \texttt{arXiv:2211.13876}.


\end{thebibliography}
\end{document}